\documentclass[12pt,reqno]{amsart}
\usepackage[margin=1in]{geometry}
\usepackage{amsmath,amssymb,amsthm,graphicx,amsxtra, setspace}
\usepackage[utf8]{inputenc}
\usepackage{mathrsfs}
\usepackage{hyperref}
\usepackage{upgreek}
\usepackage{mathtools}
\usepackage[mathcal]{euscript}
\usepackage{xcolor}
\allowdisplaybreaks

\usepackage[pagewise]{lineno}

\usepackage{graphicx,eurosym}
\usepackage{hyperref}
\usepackage{mathtools}

\usepackage[cyr]{aeguill}

\colorlet{darkblue}{blue!50!black}

\hypersetup{
	colorlinks,%
	citecolor=blue,%
	filecolor=red,%
	linkcolor=darkblue,%
	urlcolor=blue,%
	pdfnewwindow=true,%
	pdfstartview={FitH}
}

\usepackage{graphicx,amscd,mathrsfs,wrapfig,mathrsfs,lipsum}
\usepackage{eufrak}
\usepackage{tikz}
\usepackage{multicol}
\usepackage{caption}
\usetikzlibrary{arrows}

\colorlet{darkblue}{blue!50!black}

\makeatletter

\renewcommand{\tocsection}[3]{%
	\indentlabel{\@ifnotempty{#2}{\bfseries\ignorespaces#1 #2\quad}}\bfseries#3}
\renewcommand{\tocsubsection}[3]{%
	\indentlabel{\@ifnotempty{#2}{\ignorespaces#1 #2\quad}}#3}
 \let\oldtocsubsubsection=\tocsubsubsection
 \renewcommand{\tocsubsubsection}[2]{\hspace{2em}\oldtocsubsubsection{#1}{#2}}

\newcommand\@dotsep{4.5}
\def\@tocline#1#2#3#4#5#6#7{\relax
	\ifnum #1>\c@tocdepth % then omit
	\else
	\par \addpenalty\@secpenalty\addvspace{#2}%
	\begingroup \hyphenpenalty\@M
	\@ifempty{#4}{%
		\@tempdima\csname r@tocindent\number#1\endcsname\relax
	}{%
		\@tempdima#4\relax
	}%
	\parindent\z@ \leftskip#3\relax \advance\leftskip\@tempdima\relax
	\rightskip\@pnumwidth plus1em \parfillskip-\@pnumwidth
	#5\leavevmode\hskip-\@tempdima{#6}\nobreak
	\leaders\hbox{$\m@th\mkern \@dotsep mu\hbox{.}\mkern \@dotsep mu$}\hfill
	\nobreak
	\hbox to\@pnumwidth{\@tocpagenum{\ifnum#1=1\bfseries\fi#7}}\par% <-- \bfseries for \section page
	\nobreak
	\endgroup
	\fi}
\AtBeginDocument{%
	\expandafter\renewcommand\csname r@tocindent0\endcsname{0pt}
}
\def\l@subsection{\@tocline{2}{0pt}{2.5pc}{5pc}{}}
\makeatother

\newtheorem{theorem}{Theorem}[section]
\newtheorem{lemma}[theorem]{Lemma}
\newtheorem{proposition}[theorem]{Proposition}

\newtheorem{definition}[theorem]{Definition}

\newtheorem{remark}[theorem]{Remark}

\newtheorem{hypothesis}[theorem]{Hypothesis}

\let\originalleft\left
\let\originalright\right
\renewcommand{\left}{\mathopen{}\mathclose\bgroup\originalleft}
\renewcommand{\right}{\aftergroup\egroup\originalright}

\renewcommand{\d}{\/\mathrm{d}\/}

\def\w{\textbf{W}^{\varepsilon}_{{\theta}^{\varepsilon}}}

\def\L{\mathbb{L}}
\def\A{\mathrm{A}}
\def\I{\mathrm{I}}
\def\F{\mathrm{F}}
\def\C{\mathrm{C}}
\def\f{\boldsymbol{f}}

\def\B{\mathrm{B}}
\def\D{\mathrm{D}}

\def\X{\mathbb{X}}
\def\x{\boldsymbol{x}}

\def\h{\boldsymbol{h}}
\def\z{\boldsymbol{z} }
\def\v{\boldsymbol{v}}
\def\w{\boldsymbol{w}}
\def\W{\mathrm{W}}

\def\Q{\mathrm{Q}}

\def\N{\mathbb{N}}

\def\V{\mathbb{V}}
\def\wi{\widetilde}
\def\Q{\mathrm{Q}}
\def\u{\mathrm{U}}
\def\P{\mathrm{P}}
\def\u{\boldsymbol{u}}
\def\H{\mathbb{H}}

\newcommand{\R}{\mathbb{R}}

\renewcommand{\d}{\/\mathrm{d}\/}

\newcommand{\Addresses}{{% additional braces for segregating \footnotesize
		\footnote{
			\noindent \textsuperscript{1,3}Department of Mathematics, Indian Institute of Technology Roorkee-IIT Roorkee,
		Haridwar Highway, Roorkee, Uttarakhand 247667, INDIA.
		
		\noindent \textsuperscript{2}School of Mathematical Sciences, Guizhou Normal University, Guiyang 550001, China.
		\par\nopagebreak
		\noindent  \textit{e-mail:} \texttt{Manil T. Mohan: \href{maniltmohan@ma.iitr.ac.in}{maniltmohan@ma.iitr.ac.in}, \href{maniltmohan@gmail.com}{maniltmohan@gmail.com}.}
		
		\textit{e-mail:} \texttt{Renhai Wang: \href{rwang-math@outlook.com}{rwang-math@outlook.com}.}
		
		\textit{e-mail:} \texttt{Kush Kinra: \href{kkinra@ma.iitr.ac.in}{kkinra@ma.iitr.ac.in}.}
		
		\noindent \textsuperscript{*}Corresponding author.
			
			\textit{Key words:} Backward compact random attractor, asymptotic autonomy, stochastic Navier-Stokes equations, backward tempered set, backward asymptotic compactness, backward flattening estimate.
			
			Mathematics Subject Classification (2020): Primary 37L55; Secondary 37B55, 35B41, 35B40.

}}}

\begin{document}
	%	\linenumbers
	
	\title[Asymptotic autonomy  for random attractors for 2D SNSE equations]{Asymptotic Autonomy of Random Attractors in Regular Spaces for
Non-autonomous Stochastic Navier-Stokes Equations
		\Addresses}
	
	\author[K. Kinra, Renhai Wang and M. T. Mohan]
{Kush Kinra\textsuperscript{1}, Renhai Wang\textsuperscript{2} and Manil T. Mohan\textsuperscript{3*}}

	\maketitle
	
	\begin{abstract}
		This article concerns the long-term  random dynamics in regular spaces for a non-autonomous Navier-Stokes equation defined on a bounded smooth domain $\mathcal{O}$ driven by multiplicative and additive noise.  For the two kinds of noise driven equations, we demonstrate  the existence of a unique pullback attractor which is backward compact and asymptotically autonomous in $\mathbb{L}^2(\mathcal{O})$ and $\mathbb{H}_0^1(\mathcal{O})$, respectively.  The backward-uniform flattening property of the solution is used to prove the backward-uniform pullback asymptotic compactness of the non-autonomous random dynamical systems in the regular space $\mathbb{H}_0^1(\mathcal{O})$.
	\end{abstract}
	
\tableofcontents
	
	\section{Introduction} \label{sec1}\setcounter{equation}{0}
		In this work, we are interested in the  \emph{existence and asymptotic autonomy of random attractors} of the mathematical model concerning the two dimensional stochastic Navier-Stokes equations (SNSE) driven by multiplicative as well as additive noises with non-autonomous forcing term (deterministic term) defined on bounded domains. Let $\mathcal{O}$ be an open, connected and bounded subset of $\R^2$, the boundary of which is of  class $\mathrm{C}^2$.  Given $\tau\in\R$, consider the following 2D non-autonomous SNSE in $\mathcal{O}$:
		\begin{equation}\label{1}
			\left\{
			\begin{aligned}
				\frac{\partial \u}{\partial t}-\nu \Delta\u+(\u\cdot\nabla)\u+\nabla p&=\boldsymbol{f}+S(\u)\circ\frac{\d \W}{\d t}, \ \ \  \text{ in }\  \mathcal{O}\times(\tau,+\infty), \\ \nabla\cdot\u&=0, \hspace{30mm}  \text{ in } \ \ \mathcal{O}\times(\tau,+\infty), \\ \u&=0,\hspace{30mm}  \text{ in } \ \ \partial\mathcal{O}\times(\tau,+\infty), \\
				\u(x,\tau)&=\u_0(x),\hspace{23mm} x\in \mathcal{O} \ \text{ and }\ \tau\in\R,
			\end{aligned}
			\right.
		\end{equation}
		where $\u(x,t) :\mathcal{O}\times(\tau,+\infty)\to \R^2$ stands for the velocity field, $p(x,t):	\mathcal{O}\times(\tau,+\infty)\to\R$ denotes the pressure field and $\f(x,t):	\mathcal{O}\times(\tau,+\infty)\to \R^2$ is an external forcing. The constant $\nu>0$ is the \emph{kinematic viscosity} coefficient of the fluid. Here $S(\u)=\u$ (multiplicative noise) or independent of $\u$ (additive noise) is the diffusion term, the symbol $\circ$ means that the stochastic integral should be understood in the sense of Stratonovich and $\W=\W(t,\omega)$ is an one-dimensional two-sided Wiener process defined on some filtered probability space $(\Omega, \mathscr{F}, (\mathscr{F})_{t\in\R}, \mathbb{P})$ (see Subsection \ref{2.5}). The theory of global attractors for the deterministic dynamical systems is very well investigated in literature, see \cite{ICh,Robinson3,Robinson4,R.Temam} etc and references therein. For the well-posedness, global attractors and their properties (such as fractal dimension and upper semicontinuity, etc.) of 2D deterministic NSE, see \cite{CLR2,FSX,FMRT,Robinson2,R.Temam,ZD}, etc. and references therein. For the unique solvability of 2D SNSE, we refer the readers to \cite{GZ,MS,SS}, etc. and references therein. %The unique solvability of 3D Navier-Stokes equations is still a challenging  open problem (cf. \cite{Robinson1}).  %It is one of the seven Clay Millennium Prize Problems, the solution of which (either positive or negative) will be awarded with a prize of one million dollars \cite{Robinson1} etc.
	
		Generally, the existence of random attractors  for stochastic systems is based on some  transformations which convert the stochastic system into a pathwise deterministic system (cf. \cite{BCF,CDF,CF,LaRo,Robinson5}, etc for the concept of random attractors and their properties for compact random dynamical systems). In \cite{SandN_Wang}, the author established  a necessary and sufficient criteria for the existence of pullback attractors for non-compact random dynamical systems. In the literature, for SNSE related models, this transformation is available only when the white noise is either linear multiplicative or additive, see  \cite{BCLLLR,CFL,FY,GLS,KM3,X.LI,LG,PeriodicWang,chenp,rwang1,rwang2,Zhang,Zhang1}, etc and references therein. In order to deal with the nonlinear diffusion term, the concept of weak pullback mean random attractors was introduced in \cite{Wang} and applied to physically relevant models (cf. \cite{KM4,Wang1,Wang3,Wang4,Wang5}, etc), where the authors assumed diffusion term to be Lipschitz nonlinear. Another approach in the direction of random attractors, when the diffusion term is nonlinear, is the Wong-Zakai approximation of random attractors, see \cite{GGW,GLW,KM7}, etc and references therein.
	
	In general, a non-autonomous random attractor carries the form
	\begin{align*}
		\mathcal{A}_{\varsigma}=\{\mathcal{A}_{\varsigma}(\tau,\omega):\tau\in\R, \omega\in\Omega\},
	\end{align*}
	 where $\varsigma$ stands for some external perturbation parameter. In the literature, the robustness with respect to the external parameter $\varsigma$ of the random attractors $\mathcal{A}_{\varsigma}(\tau,\omega)$ for SNSE have been established  (cf. \cite{HCPEK,GGW,GLW}, etc) but not for the internal parameter $\tau$. The aim of this work is not only to prove the existence a of random attractor $\mathcal{A}(\tau,\omega)$, but also the asymptotic autonomy of $\mathcal{A}(\tau,\omega)$, that is,
	 \begin{align}\label{AA}
	 	\lim_{\tau\to-\infty}\text{dist}_{\X}\left(\mathcal{A}(\tau,\omega),\mathcal{A}_{\infty}(\omega)\right)=0,\  \mathbb{P}\text{-a.s. } \omega\in\Omega.
	 \end{align}
Here, $\text{dist}_{\mathbb{X}}(\cdot,\cdot)$ denotes the Hausdorff semi-distance between two non-empty subsets of some Banach space $\mathbb{X}$, that is, for non-empty sets $A,B\subset \mathbb{X}$ $$\text{dist}_{\mathbb{X}}(A,B)=\sup_{a\in A}\inf_{b\in B} d(a,b)$$ and  $\mathcal{A}_{\infty}(\omega)$ is the random attractor of the corresponding autonomous system.

To prove the asymptotic autonomy of random attractors, we need to establish  the existence of backward compact random attractors. For this purpose, we define a backward tempered attracting universe (see \eqref{D-NSE}) which is indeed smaller than the usual tempered attracting universe (by omitting supremum from \eqref{D-NSE}). It is worth mentioning  that the radii of the absorbing set in this case is taken as the supremum over an uncountable set $(-\infty,\tau]$ (see Proposition \ref{IRAS}) which causes difficulties to show the measurability of absorbing sets. In \cite{YR}, the authors used Egoroff and Lusin theorems to deal with this difficulty (measurability of absorbing sets), see Proposition 3.1 in \cite{YR}. Moreover, the authors in \cite{YR} (Section 5) demonstrated abstract results for asymptotic autonomy of random attractors. For asymptotic autonomy of random attractors, we refer the readers to \cite{CGTW,YR,ZL}, etc. To the best of our knowledge,  there is no result available in the literature on the existence of backward compact random attractors as well as for their asymptotic autonomy for 2D SNSE driven by multiplicative noise. Whereas, for 2D SNSE driven by additive noise, the existence of backward compact random attractors in $\H$ as well as in $\V$ is known (cf. \cite{LXY}), but the asymptotic autonomy results not available. Therefore, in  this work, we  demonstrate the asymptotic autonomy of random attractors in both the spaces for additive noise case.

As discussed earlier, the backward asymptotic compactness plays a key role in our analysis. First, we establish the existence of increasing random absorbing sets in $\H$ and $\V$. We obtain the backward asymptotic compactness in $\H$ using compact Sobolev embedding $\V \subset\H$. Since the existence of random absorbing sets in $\D(\A^s)$, $s>1/2$ is not available for the 2D SNSE, the compactness arguments cannot be used to obtain the backward asymptotic compactness in $\mathbb{V}$. Therefore, in order to prove the backward asymptotic compactness in $\V$, we prove that the solution to the 2D SNSE driven by multiplicative noise satisfies the backward flattening property in $\V$. Even though, it has been established in the literature that there exists a unique random pullback attractor for SNSE in unbounded Poincar\'e domains (cf. \cite{BCLLLR,BL,PeriodicWang} etc), but the existence of backward compact random attractors and their asymptotic autonomy in unbounded domains is still an interesting open problem which will be addressed in a future work.
\vskip 0.2 cm
\noindent
\textbf{Aims and scopes of the work:} The major aims and novelties of this work are:
\begin{itemize}
	\item [(i)] We show the existence and asymptotic autonomy of random attractors for 2D SNSE  \eqref{1} driven by multiplicative noise on bounded domains in $\H$ as well as in $\V$ (Theorem \ref{MT1}).
	\item [(ii)] We show the asymptotic autonomy of random attractors for 2D SNSE \eqref{1} driven by additive noise on bounded domains under the additional Hypothesis \ref{AonH-N} in $\H$ as well as in $\V$ (Theorem \ref{MT1-N}).
\end{itemize}

The arrangement of the further sections is as follows: In the next section, we explain the necessary functions spaces, and the linear and bilinear operators along with their properties which are required for the analysis of this work. Then we provide an  abstract formulation of \eqref{1} using the linear and bilinear operators. In the same section, we discuss the Ornstein-Uhlenbeck process along with its properties which help us to transform the stochastic system \eqref{1} into a pathwise deterministic system. In section \ref{sec3}, we consider the 2D SNSE \eqref{1} driven by multiplicative noise and convert it into a pathwise deterministic system \eqref{CNSE-M} with the help of Ornstein-Uhlenbeck process. In order to apply the abstract results established in \cite{YR}, one needs to verify all the assumptions of Theorem \ref{Abstract-result}. For that purpose, we first prove the Lusin continuity (Lemma \ref{LusinC}), convergence of non-autonomous random dynamical system (RDS) to the corresponding autonomous RDS as $\tau\to-\infty$ in $\H$ (Proposition \ref{Back_conver}), increasing random absorbing set in $\H$ (Proposition \ref{IRAS}) and backward asymptotic compactness (using compact Sobolev embeddings). The above results altogether demonstrates the first main result of  section \ref{sec3} (Theorem \ref{MT1}) using the abstract result established in \cite{YR}, that is, the existence and asymptotic autonomy of random attractors for 2D SNSE driven by multiplicative noise in $\H$. Applying similar arguments, we demonstrate the existence and asymptotic autonomy of random attractors for 2D SNSE driven by multiplicative noise in $\V$ (Theorem \ref{MT1-V}), where we use the backward flattening property to prove the backward asymptotic compactness. In the  final section, we consider the 2D SNSE \eqref{1} driven by additive noise under the extra Hypothesis \ref{AonH-N}. Since the existence of backward compact random attractors in $\H$ as well as $\V$ is established in \cite{LXY}, we only prove the asymptotic autonomy of random attractors in both spaces. Therefore, we prove the convergence of non-autonomous RDS to the corresponding autonomous RDS as $\tau\to-\infty$ in $\H$ and $\V$ (Propositions \ref{Back_conver-N} and \ref{Back_conver-VA}, respectively) and using the abstract result established in \cite{YR}, we establish the main result of section \ref{sec4} (Theorem \ref{MT1-N}).

		\section{Mathematical formulation}\label{sec2}\setcounter{equation}{0}
	This section is devoted for providing the necessary function spaces needed to obtain the results of this work. Further, we define linear and bilinear operators to obtain abstract formulation of the system \eqref{1}. Finally, we discuss the Ornstein-Uhlenbeck process, its properties and the backward tempered random set.
	\subsection{Function spaces and operators}
	Let us define the space $$\mathcal{V}:=\{\u\in\C_0^{\infty}(\mathcal{O};\R^2):\nabla\cdot\u=0\},$$ where $\C_0^{\infty}(\mathcal{O};\R^2)$ denote the space of all infinitely differentiable functions  ($\R^2$-valued) with compact support in $\mathcal{O}$. Let $\H$, $\V$ and $\wi\L^p$, for $p\in(2,+\infty)$ denote the completion of $\mathcal{V}$ in 	$\mathrm{L}^2(\mathcal{O};\R^2)$, $\mathrm{H}_0^1(\mathcal{O};\R^2)$ and $\mathrm{L}^p(\mathcal{O};\R^2)$ norms, respectively. The spaces  $\H$, $\V$ and $\wi\L^p$  are endowed with the norms $\|\u\|_{\H}^2:=\int_{\mathcal{O}}|\u(x)|^2\d x$, $\|\u\|_{\V}^2:=\int_{\mathcal{O}}|\nabla\u(x)|^2\d x$ and  $\|\u\|_{\wi \L^p}^2:=\int_{\mathcal{O}}|\u(x)|^p\d x,$ respectively. The induced duality between the spaces $\V$ and $\V'$, and $\widetilde{\L}^p$ and its dual $\widetilde{\L}^{\frac{p}{p-1}}$ is denoted by $\langle\cdot,\cdot\rangle.$ Moreover, we have the following continuous  embedding also:
	$$\V\hookrightarrow\H\equiv\H'\hookrightarrow\V'.$$	
		\subsubsection{Linear operator}\label{LO}
	Let $\mathcal{P}: \L^2(\mathcal{O}) \to\H$ denote the Helmholtz-Hodge orthogonal projection (cf. \cite{OAL}). Let us define the Stokes operator
	\begin{equation*}
		\A\u:=-\mathcal{P}\Delta\u,\;\u\in\D(\A).
	\end{equation*}
	The operator $\A$ is a linear continuous operator from $\V$ into $\V'$, satisfying
	\begin{equation*}
		\langle\A\u,\v\rangle=(\!(\u,\v)\!), \ \ \ \u,\v\in\V.
	\end{equation*}
	Since the boundary of $\mathcal{O}$ is of class $\mathrm{C}^2$, this infer that $\D(\A)=\V\cap\H^2(\mathcal{O})$ and $\|\A\u\|_{\H}$ defines a norm in $\D(\A),$ which is equivalent to the one in $\H^2(\mathcal{O})$ (cf. Cattabriga regularity theorem, \cite{Temam2}). Note that the operator $\A$ is a non-negative self-adjoint operator in $\H$ and
	\begin{align}\label{2.7a}
		\langle\A\u,\u\rangle =\|\u\|_{\V}^2,\ \textrm{ for all }\ \u\in\V, \ \text{ so that }\ \|\A\u\|_{\V'}\leq \|\u\|_{\V}.
	\end{align}
For the bounded domain $\mathcal{O}$, we also have (Poincar\'e inequality)
\begin{align}\label{poin}
	\lambda_1\|\u\|_{\mathbb{H}}^2\leq\|\u\|_{\mathbb{V}}^2,  \text{ for all }  \u\in\V,
\end{align}where $\lambda_1$ is the smallest eigenvalue of operator $\A$.
	\subsubsection{Bilinear operator}\label{BO}
	Let us define the \emph{trilinear form} $b(\cdot,\cdot,\cdot):\V\times\V\times\V\to\R$ by $$b(\u,\v,\w)=\int_{\mathcal{O}}(\u(x)\cdot\nabla)\v(x)\cdot\w(x)\d x=\sum_{i,j=1}^2\int_{\mathcal{O}}\u_i(x)\frac{\partial \v_j(x)}{\partial x_i}\w_j(x)\d x.$$ If $\u, \v$ are such that the linear map $b(\u, \v, \cdot) $ is continuous on $\V$, the corresponding element of $\V'$ is denoted by $\B(\u, \v)$. We also denote $\B(\u) = \B(\u, \u)=\mathcal{P}[(\u\cdot\nabla)\u]$.	An integration by parts gives
	\begin{equation}\label{b0}
		\left\{
		\begin{aligned}
			b(\u,\v,\v) &= 0,\ \text{ for all }\ \u,\v \in\V,\\
			b(\u,\v,\w) &=  -b(\u,\w,\v),\ \text{ for all }\ \u,\v,\w\in \V.
		\end{aligned}
		\right.\end{equation}
	\begin{remark}
		The following estimates on the trilinear form $b(\cdot,\cdot,\cdot)$ is used in the sequel $($see Chapter 2, section 2.3, \cite{Temam1}$):$
			\begin{itemize}
			\item [$\bullet$] for all $\u, \v, \w\in \V$,
			\begin{align}
				|b(\u,\v,\w)|&\leq
				C\|\u\|^{1/2}_{\H}\|\u\|^{1/2}_{\V}\|\v\|_{\V}\|\w\|^{1/2}_{\H}\|\w\|^{1/2}_{\V}.   \label{b1}
			\end{align}
			\item  [$\bullet$] for all $\u_1\in \V, \u_2\in \D(\A), \u_3\in \H$, we have
			\begin{align}
				|b(\u_1,\u_2,\u_3)|&\leq C\|\u_1\|^{1/2}_{\H}\|\u_1\|^{1/2}_{\V}\|\u_2\|^{1/2}_{\V}\|\A\u_2\|^{1/2}_{\H}\|\u_3\|_{\H}.\label{b2}
			\end{align}
		\end{itemize}
	\end{remark}
	\begin{remark}
		Note that $\langle\B(\u,\u-\v),\u-\v\rangle=0$, which  implies that
		\begin{equation}\label{441}
			\begin{aligned}
				\langle \B(\u)-\B(\v),\u-\v\rangle =\langle\B(\u-\v,\v),\u-\v\rangle=-\langle\B(\u-\v,\u-\v),\v\rangle.
			\end{aligned}
		\end{equation}
	\end{remark}
\subsection{Abstract formulation and Ornstein-Uhlenbeck process}\label{2.5}
Taking the projection $\mathcal{P}$ on SNSE equations \eqref{1}, we write
\begin{equation}\label{SNSE}
	\left\{
	\begin{aligned}
		\frac{\d\u(t)}{\d t}+\nu \A\u(t)+\B(\u(t))&=\f(t) +S(\u(t))\circ\frac{\d \W(t)}{\d t} , \\
		\u(x,\tau)&=\u_{0}(x),\ \ \	x\in \mathcal{O},
	\end{aligned}
	\right.
\end{equation}
	where $S(\u)=\u$ or independent of $\u$ (for simplicity of notations, we denoted $\mathcal{P}S(\u)$ as $S(\u)$ and $\mathcal{P}\f$ as $\f$). Here, $\W(t,\omega)$ is the standard scalar Wiener process on the probability space $(\Omega, \mathscr{F}, \mathbb{P}),$ where $$\Omega=\{\omega\in C(\R;\R):\omega(0)=0\},$$ endowed with the compact-open topology given by the complete metric
\begin{align*}
	d_{\Omega}(\omega,\omega'):=\sum_{m=1}^{\infty} \frac{1}{2^m}\frac{\|\omega-\omega'\|_{m}}{1+\|\omega-\omega'\|_{m}},\text{ where } \|\omega-\omega'\|_{m}:=\sup_{-m\leq t\leq m} |\omega(t)-\omega'(t)|,
\end{align*}
and $\mathscr{F}$ is the Borel sigma-algebra induced by the compact-open topology of $(\Omega,d_{\Omega}),$ $\mathbb{P}$ is the two-sided Wiener measure on $(\Omega,\mathscr{F})$. From \cite{FS}, it is clear that  the measure $\mathbb{P}$ is ergodic and invariant under the translation-operator group $\{\theta_t\}_{t\in\R}$ on $\Omega$ defined by
\begin{align*}
	\theta_t \omega(\cdot) = \omega(\cdot+t)-\omega(t), \ \text{ for all }\ t\in\R, \ \omega\in \Omega.
\end{align*}
The operator $\theta(\cdot)$ is known as \emph{Wiener shift operator}. Moreover, the quadruple $(\Omega,\mathscr{F},\mathbb{P},\theta)$ defines a metric dynamical system, see \cite{Arnold,BCLLLR}.
\subsubsection{Ornstein-Uhlenbeck process}
Consider for some $\sigma>0$ (which will be specified later)
\begin{align}\label{OU1}
	z(\theta_{t}\omega) =  \int_{-\infty}^{t} e^{-\sigma(t-\xi)}\d \W(\xi), \ \ \omega\in \Omega,
\end{align} which is the stationary solution of the one dimensional Ornstein-Uhlenbeck equation
\begin{align}\label{OU2}
	\d z(\theta_t\omega) + \sigma z(\theta_t\omega)\d t =\d\W(t).
\end{align}
It is known from \cite{FAN} that there exists a $\theta$-invariant subset $\widetilde{\Omega}\subset\Omega$ of full measure such that $z(\theta_t\omega)$ is continuous in $t$ for every $\omega\in \widetilde{\Omega},$ and
\begin{align}
	\mathbb{E}\left(|z(\theta_s\omega)|^{\xi}\right)&=\frac{\Gamma\left(\frac{1+\xi}{2}\right)}{\sqrt{\pi\sigma^{\xi}}}, \ \text{ for all } \xi>0, s\in \R,\label{Z2}\\
	\lim_{t\to \pm \infty} \frac{|z(\theta_t\omega)|}{|t|}&=0,\label{Z3}\\
	\lim_{t\to \pm \infty} \frac{1}{t} \int_{0}^{t} z(\theta_{\xi}\omega)\d\xi &=0,\label{Z4}\\
	\lim_{t\to +\infty} e^{-\delta t}|z(\theta_{-t}\omega)| &=0, \ \text{ for all } \ \delta>0,\label{Z5}
\end{align} where $\Gamma$ is the Gamma function. For further analysis of this work, we do not distinguish between $\widetilde{\Omega}$ and $\Omega$.

Since, $\omega(\cdot)$ has sub-exponential growth  (cf. Lemma 11, \cite{CGSV}), $\Omega$ can be written as $\Omega=\cup_{N\in\N}\Omega_{N}$, where
\begin{align*}
	\Omega_{N}:=\{\omega\in\Omega:|\omega(t)|\leq Ne^{|t|},\text{ for all }t\in\R\}, \text{ for all } \ N\in\N.
\end{align*}
\begin{lemma}\label{conv_z}
	For each $N\in\N$, suppose $\omega_k,\omega_0\in\Omega_{N}$ such that $d_{\Omega}(\omega_k,\omega_0)\to0$ as $k\to+\infty$. Then, for each $\tau\in\R$ and $T\in\R^+$ ,
	\begin{align}
		\sup_{t\in[\tau,\tau+T]}&\bigg[|z(\theta_{t}\omega_k)-z(\theta_{t}\omega_0)|+|e^{ z(\theta_{t}\omega_k)}-e^{ z(\theta_{t}\omega_0)}|\bigg]\to 0 \ \text{ as } \ k\to+\infty,\nonumber\\
		\sup_{k\in\N}\sup_{t\in[\tau,\tau+T]}&|z(\theta_{t}\omega_k)|\leq C(\tau,T,\omega_0).\label{conv_z2}
	\end{align}
\end{lemma}
\begin{proof}
	See Corollary 22 and Lemma 2.5 in \cite{CLL} and \cite{YR}, respectively.
\end{proof}
\subsubsection{Backward tempered random set}{\cite{YR}}
A bi-parametric set $\mathcal{D}=\{\mathcal{D}(\tau,\omega)\}$ in a Banach space $\X$ is said to be \emph{backward tempered} if
\begin{align}\label{BackTem}
	\lim_{t\to +\infty}e^{-ct}\sup_{s\leq \tau}\|\mathcal{D}(s-t,\theta_{-t}\omega)\|^2_{\X}=0\ \text{ for all }\  (\tau,\omega)\in\R\times\Omega,
\end{align}
where $c>0$ is a constant and $\|\mathcal{D}\|_{\X}=\sup\limits_{\x\in \mathcal{D}}\|\x\|_{\X}.$
\subsubsection{Class of random sets}
\begin{itemize}
	\item Let ${\mathfrak{D}}$ be the collection of all backward tempered subsets of $\H$, that is,
	\begin{align}\label{D-NSE}
		{\mathfrak{D}}=\left\{{\mathcal{D}}=\{{\mathcal{D}}(\tau,\omega):(\tau,\omega)\in\R\times\Omega\}:\lim_{t\to +\infty}e^{-\frac{\nu\lambda_{1}}{3}t}\sup_{s\leq \tau}\|{\mathcal{D}}(s-t,\theta_{-t}\omega)\|^2_{\H}=0\right\},
	\end{align}
	where $\nu$ and $\lambda_{1}$ are same as given in  \eqref{1} and \eqref{poin}, respectively.
	\item Let ${\mathfrak{D}}_{\infty}$ be the collection of all tempered subsets of $\H$, that is,
	\begin{align*}
		{\mathfrak{D}}_{\infty}=\left\{\widehat{\mathcal{D}}=\{\widehat{\mathcal{D}}(\omega):\omega\in\Omega\}:\lim_{t\to +\infty}e^{-\frac{\nu\lambda_{1}}{3}t}\|\widehat{\mathcal{D}}(\theta_{-t}\omega)\|^2_{\H}=0\right\},
	\end{align*}
	where $\nu$ and $\lambda_{1}$ are same as in  \eqref{D-NSE}.
\end{itemize}
\subsection{Abstract results for asymptotic autonomy}
This subsection provides  abstract results for the asymptotic autonomy of random attractors introduced in \cite{YR}. Let $\varphi$ be a general NRDS on a Banach space $\mathbb{U}$ over $(\Omega,\mathscr{F},\mathbb{P},\theta),$ see Proposition \ref{NRDS} below. Let $\mathfrak{D}=\{\mathcal{D}(\tau,\omega)\}$ be the universe of some bi-parametric sets. We assume that $\mathfrak{D}$ is backward-closed, which means $\widetilde{\mathcal{D}}\in\mathfrak{D}$ provided $\mathcal{D}\in\mathfrak{D}$ and $\widetilde{\mathcal{D}}=\cup_{s\leq\tau}\mathcal{D}(s,\omega)$. Also, $\mathfrak{D}$ is inclusion closed (cf. \cite{YR,SandN_Wang}). Note that the universe as given in \eqref{BackTem} is both backward-closed and inclusion-closed.
\begin{definition}\label{BCRA}
	A bi-parametric set $\mathcal{A}=\{\mathcal{A}(\tau,\omega):\tau\times\omega\in(\R,\Omega)\}$ is said to be a $\mathfrak{D}$-backward compact random attractor for a NRDS $\varphi$ if
	\begin{enumerate}
		\item $\mathcal{A}\in\mathfrak{D}$ and $\mathcal{A}$ is a random set,
		\item $\mathcal{A}$ is backward compact, that is, both $\mathcal{A}(\tau,\omega)$ and $\overline{\cup_{s\leq\tau}\mathcal{A}(s,\omega)}$ are compact,
		\item $\mathcal{A}$ is invariant, that is, $\varphi(t,\tau,\omega)\mathcal{A}(\tau,\omega)=\mathcal{A}(t+\tau,\theta_{t}\omega)$ for $t\geq0$,
		\item $\mathcal{A}$ is attracting under the Hausdorff semi-distance, that is, for each $\mathcal{D}\in \mathfrak{D}$,
		\begin{align*}
			\lim_{t\to+\infty}\mathrm{dist}_{\mathbb{U}}(\varphi(t,\tau-t,\theta_{-t}\omega)\mathcal{D}(\tau-t,\theta_{-t}\omega),\mathcal{A}(\tau,\omega))=0.
		\end{align*}
	\end{enumerate}
\end{definition}
\begin{definition}\label{BAC}
	A non-autonomous cocycle $\varphi$ on $\mathbb{U}$ is said to be $\mathfrak{D}$-backward asymptotically compact if for each $(\tau,\omega,\mathcal{D})\in\R\times\Omega\times\mathfrak{D}$, the sequence $$\{\varphi(t_n,s_n-t_n,\theta_{-t_n}\omega)x_n\}_{n=1}^{+\infty}$$ has a convergent subsequence in $\mathbb{U}$, whenever $s_n\leq\tau$, $t_n\to+\infty$ and $x_n\in\mathcal{D}(s_n-t_n,\theta_{-t_{n}}\omega)$.
\end{definition}
Let us now provide the abstract result for asymptotic autonomy and backward compactness of random attractors which was proved in \cite{YR}. Let $\mathfrak{D}_{\infty}=\{D(\omega)\}$ be an inclusion-closed universe of some single-parametric sets.
\begin{theorem}[Theorem 5.1, \cite{YR}]\label{Abstract-result}
	Assume that an NRDS $\varphi$ satisfies the following two conditions:
	\begin{itemize}
		\item [($a_1$)] $\varphi$ has a closed random absorbing set $\mathcal{K}\in\mathfrak{D}$;
		\item [($a_2$)] $\varphi$ is $\mathfrak{D}$-backward asymptotically compact.
	\end{itemize}
Then, $\varphi$ has a unique backward compact random attractor $\mathcal{A}\in\mathfrak{D}$. Moreover, let $\varphi_{\infty}$ be an RDS with a $\mathfrak{D}_{\infty}$-random attractor $\mathcal{A}_{\infty}$, and suppose that
	\begin{itemize}
		\item [($b_1$)] $\varphi$ backward converges to $\varphi_{\infty}$ in the following sense:
		\begin{align*}
			\lim_{\tau\to-\infty}\|\varphi(t,\tau,\omega)x_{\tau}-\varphi_{\infty}(t,\omega)x_0\|_{\mathbb{U}}=0, \text{ for all } t\geq 0 \text{ and } \omega\in\Omega,
		\end{align*}
	whenever $\|x_{\tau}-x_0\|_{\mathbb{U}}\to 0$ as $\tau\to-\infty$.
	\item [($b_2$)] $\mathcal{K}_{\tau_0}\in\mathfrak{D}_{\infty}$ for some $\tau_{0}<0$, where $\mathcal{K}_{\tau_0}(\omega):=\bigcup_{\tau\leq\tau_0}\mathcal{K}(\tau,\omega)$.
	\end{itemize}
Then, $\mathcal{A}$ backward converges to $\mathcal{A}_{\infty}:$
	\begin{align*}
	\lim_{\tau\to -\infty}\mathrm{dist}_{\mathbb{U}}(\mathcal{A}(\tau,\omega),\mathcal{A}_{\infty}(\omega))=0\  \emph{ for all } \ \omega\in\Omega,
\end{align*}
where $\mathcal{A}_{\infty}$ is the random attractor for the corresponding autonomous system. For any discrete sequence $\tau_{n}\to-\infty$ as $n\to+\infty$, there is a subsequence $\{\tau_{n_k}\}\subseteq\{\tau_{n}\}$  such that
\begin{align*}
	\lim_{k\to +\infty}\mathrm{dist}_{\mathbb{U}}(\mathcal{A}(\tau_{n_k},\theta_{\tau_{n_k}}\omega),\mathcal{A}_{\infty}(\theta_{\tau_{n_k}}\omega))=0\  \emph{ for all } \ \omega\in\Omega.
\end{align*}
\end{theorem}

\section{2D SNSE: Multiplicative noise}\label{sec3}\setcounter{equation}{0}
In this section, we consider the 2D SNSE equations driven by a multiplicative white noise ($S(\u)=\u$ in \eqref{1}) and establish the existence and asymptotic autonomy of random attractors. Let us define $$\v(t,\tau,\omega,\v_{\tau})=e^{-z(\theta_{t}\omega)}\u(t,\tau,\omega,\u_{\tau})\ \text{ with  }\  \v_{\tau}=e^{-z(\theta_{\tau}\omega)}\u_{\tau},$$ where $z$ satisfies \eqref{OU2} and $\u(\cdot)$ is the solution of \eqref{SNSE} with $S(\u)=\u$. Then $\v(\cdot)$ satisfies:
\begin{equation}\label{CNSE-M}
	\left\{
	\begin{aligned}
		\frac{\d\v(t)}{\d t}+\nu \A\v(t)+e^{z(\theta_{t}\omega)}\B\big(\v(t)\big)&=\f(t) e^{-z(\theta_{t}\omega)} + \sigma z(\theta_t\omega)\v(t) , \quad t> \tau, \tau\in\R ,\\
		\v(x,\tau)&=\v_{0}(x)=e^{-z(\theta_{\tau}\omega)}\u_{0}(x), \hspace{12mm} x\in\mathcal{O},
	\end{aligned}
	\right.
\end{equation}
in $\V'$ (in weak sense). In order to prove the results of this work, the following assumption on the non-autonomous and autonomous external forcing terms $\f(\cdot,\cdot)$ and $\f_{\infty}(\cdot)$, respectively, is  needed.
\begin{hypothesis}\label{Hypo_f-N}
	The forcing terms  $\f\in\mathrm{L}^{2}_{\emph{loc}}(\R;\H)$ and $\f_{\infty}\in\H$ are such that
	\begin{align*}
		\lim_{\tau\to -\infty}\int^{\tau}_{-\infty}\|\f(t)-\f_{\infty}\|^2_{\H}\d t=0.
	\end{align*}
\end{hypothesis}
For example, one can consider  $\f(x,t)=(e^t+1)\f_0(x)$ and $\f_{\infty}=\f_0(x)$ and they satisfy the above Hypothesis \ref{Hypo_f-N} (see \cite{CGTW}). The following Lemma is adapted from the work \cite{YR}.
\begin{lemma}[Lemma 2.1, \cite{YR}]\label{Hypo_f1-N}
	Assume that Hypothesis \ref{Hypo_f-N} holds. Then, $\f$ is backward tempered, that is, for all $\gamma>0$ and $\tau\in\R$,
		\begin{align}\label{f2-N}
			\F(\gamma,\tau)=\sup_{s\leq\tau}\int_{-\infty}^{s}e^{\gamma(\xi-s)}\|\f(\xi)\|^2_{\H}\d\xi<+\infty.
		\end{align}
\end{lemma}

\subsection{Existence and asymptotic autonomy of random attractors in $\H$}
In this section, we prove the existence of backward compact random attractors and their asymptotic autonomy in $\H$. First, we obtain the results which will help us to verify the assumptions of the abstract result (Theorem \ref{Abstract-result}). The following lemma demonstrates that the system \eqref{CNSE-M} has a unique weak as well as strong solutions.
\begin{lemma}\label{Soln}
	Suppose that $\f\in\mathrm{L}^2_{\mathrm{loc}}(\R;\H)$. For each $(\tau,\omega,\v_{\tau})\in\R\times\Omega\times\H$, the system \eqref{CNSE-M} has a unique weak solution $\v(\cdot,\tau,\omega,\v_{\tau})\in\mathrm{C}([\tau,+\infty);\H)\cap\mathrm{L}^2_{\mathrm{loc}}(\tau,+\infty;\V)$ such that $\v$ is continuous with respect to initial data. In addition, for $\v_{\tau}\in \V$, there exists a unique strong solution $\v(\cdot,\tau,\omega,\v_{\tau})\in \mathrm{C}([\tau, +\infty); \V) \cap \mathrm{L}^{2}_{\mathrm{loc}}(\tau, +\infty; \D(\A))$ such that $\v$ is continuous with respect to the initial data.
\end{lemma}
\begin{proof}
	One can prove the existence and uniqueness of weak as well as strong  solutions by a standard Faedo-Galerkin approximation method (cf. \cite{R.Temam}). Also, for the continuity with respect to the initial data $\v_{\tau}$, see \cite{R.Temam}.
\end{proof}

	\subsubsection{Lusin continuity and measurability of systems}
	Lusin continuity helps us to define the  non-autonomous random dynamical system (NRDS). Next result shows the Lusin continuity of mapping with respect to $\omega\in\Omega$ of solution to the system \eqref{CNSE-M}.
\begin{proposition}\label{LusinC}
	Suppose that $\f\in\mathrm{L}^2_{\mathrm{loc}}(\R;\H)$. For each $N\in\N$, the mapping $\omega\mapsto\v(t,\tau,\omega,\v_{\tau})$ $($solution of \eqref{CNSE-M}$)$ is continuous from $(\Omega_{N},d_{\Omega_N})$ to $\H$, uniformly in $t\in[\tau,\tau+T]$ with $T>0$.
\end{proposition}
\begin{proof}
	Assume that $\omega_k,\omega_0\in\Omega_N,\ N\in\mathbb{N}$ such that $d_{\Omega_N}(\omega_k,\omega_0)\to0$ as $k\to+\infty$. Let $\mathscr{V}^k:=\v^k-\v^0,$ where $\v^k=\v(t,\tau,\omega_k,\v_{\tau})$ and $\v^0=\v(t,\tau,\omega_0,\v_{\tau})$ for $t\in[\tau,\tau+T]$. Then, $\mathscr{V}^k$ satisfies:
\begin{align}\label{LC1}
		\frac{\d\mathscr{V}^k}{\d t}&=-\nu \A\mathscr{V}^k-e^{z(\theta_{t}\omega_k)}\B\big(\v^k\big)+e^{z(\theta_{t}\omega_0)}\B\big(\v^0\big)+\f \left[e^{-z(\theta_{t}\omega_k)}-e^{-z(\theta_{t}\omega_0)}\right] \nonumber\\&\quad+ \sigma z(\theta_t\omega_k)\v^k-\sigma z(\theta_t\omega_0)\v^0\nonumber\\&=-\nu \A\mathscr{V}^k+\sigma z(\theta_t\omega_k) \mathscr{V}^k-e^{z(\theta_{t}\omega_k)}\left[\B\big(\v^k\big)-\B\big(\v^0\big)\right]-\left[e^{z(\theta_{t}\omega_k)}-e^{z(\theta_{t}\omega_0)}\right]\B\big(\v^0\big) \nonumber\\&\quad+\f \left[e^{-z(\theta_{t}\omega_k)}-e^{-z(\theta_{t}\omega_0)}\right]+\sigma\left[z(\theta_t\omega_k)-z(\theta_t\omega_0)\right]\v^0,
\end{align}
	in $\V'$ (in weak sense). Taking the inner product with $\mathscr{V}^k(\cdot)$ in \eqref{LC1}, and using \eqref{b0} and \eqref{441}, we get
	\begin{align}\label{LC2}
		\frac{1}{2}\frac{\d }{\d t}\|\mathscr{V}^k\|^2_{\H}&=-\nu\|\mathscr{V}^k\|^2_{\V}+\sigma z(\theta_t\omega_k)\|\mathscr{V}^k\|^2_{\H}-e^{z(\theta_{t}\omega_k)}b(\mathscr{V}^k,\v^0,\mathscr{V}^k)\nonumber\\&\quad+\left[e^{z(\theta_{t}\omega_k)}-e^{z(\theta_{t}\omega_0)}\right]b(\v^0,\mathscr{V}^k,\v^0)+\left[e^{-z(\theta_{t}\omega_k)}-e^{-z(\theta_{t}\omega_0)}\right](\f,\mathscr{V}^k)\nonumber\\&\quad+\sigma\left[z(\theta_t\omega_k)-z(\theta_t\omega_0)\right](\v^0,\mathscr{V}^k).
	\end{align}
Using \eqref{poin}, H\"older's and Young's inequalities, we obtain
\begin{align}
	\left|\left[e^{-z(\theta_{t}\omega_k)}-e^{-z(\theta_{t}\omega_0)}\right](\f,\mathscr{V}^k)\right|&\leq C\left|e^{-z(\theta_{t}\omega_k)}-e^{-z(\theta_{t}\omega_0)}\right|^2\|\f\|^2_{\H}+\frac{\nu\lambda_{1}}{8}\|\mathscr{V}^k\|^2_{\H}\nonumber\\&\leq C\left|e^{-z(\theta_{t}\omega_k)}-e^{-z(\theta_{t}\omega_0)}\right|^2\|\f\|^2_{\H}+\frac{\nu}{8}\|\mathscr{V}^k\|^2_{\V},\label{LC4}\\
	\left|\sigma\left[z(\theta_t\omega_k)-z(\theta_t\omega_0)\right](\v^0,\mathscr{V}^k)\right|&\leq C\left|z(\theta_t\omega_k)-z(\theta_t\omega_0)\right|^2\|\v^0\|^2_{\H}+\frac{\nu\lambda_{1}}{8}\|\mathscr{V}^k\|^2_{\H}\nonumber\\&\leq C\left|z(\theta_t\omega_k)-z(\theta_t\omega_0)\right|^2\|\v^0\|^2_{\H}+\frac{\nu}{8}\|\mathscr{V}^k\|^2_{\V}.\label{LC5}
\end{align}
 Applying \eqref{b1}, H\"older's and Young's inequalities, we estimate
\begin{align}
	\left|e^{z(\theta_{t}\omega_k)}b(\mathscr{V}^k,\v^0,\mathscr{V}^k)\right|&\leq Ce^{z(\theta_{t}\omega_k)}\|\mathscr{V}^k\|_{\H}\|\mathscr{V}^k\|_{\V}\|\v^0\|_{\V}\nonumber\\&\leq Ce^{2z(\theta_{t}\omega_k)}\|\v^0\|^2_{\V}\|\mathscr{V}^k\|^2_{\H}+\frac{\nu}{8}\|\mathscr{V}^k\|^2_{\V},\label{LC7}
\end{align}
and
\begin{align}
	\left|\left[e^{z(\theta_{t}\omega_k)}-e^{z(\theta_{t}\omega_0)}\right]b(\v^0,\mathscr{V}^k,\v^0)\right|&\leq C\left|e^{z(\theta_{t}\omega_k)}-e^{z(\theta_{t}\omega_0)}\right|\|\v^0\|_{\H}\|\mathscr{V}^k\|_{\V}\|\v^0\|_{\V}\nonumber\\&\leq C\left|e^{z(\theta_{t}\omega_k)}-e^{z(\theta_{t}\omega_0)}\right|^2\|\v^0\|^2_{\H}\|\v^0\|^2_{\V}+\frac{\nu}{8}\|\mathscr{V}^k\|^2_{\V}.\label{LC8}
\end{align}
Combining \eqref{LC2}-\eqref{LC8}, we arrive at
\begin{align}\label{LC13}
	\frac{\d }{\d t}\|\mathscr{V}^k(t)\|^2_{\H}+\frac{\nu}{2}\|\mathscr{V}^k(t)\|^2_{\V}\leq
P_1(t)\|\mathscr{V}^k(t)\|^2_{\H}+Q_1(t),
\end{align}
for a.e. $t\in[\tau,\tau+T]$, $T>0$, where
\begin{align*}
	P_1&=2z(\theta_{t}\omega_k)+Ce^{2z(\theta_{t}\omega_k)}\|\v^0\|^2_{\V},\\	Q_1&=C\left|e^{-z(\theta_{t}\omega_k)}-e^{-z(\theta_{t}\omega_0)}\right|^2\|\f\|^2_{\H}+C\left|z(\theta_t\omega_k)-z(\theta_t\omega_0)\right|^2\|\v^0\|^2_{\H}\nonumber\\&\quad+C\left|e^{z(\theta_{t}\omega_k)}-e^{z(\theta_{t}\omega_0)}\right|^2\|\v^0\|^2_{\H}\|\v^0\|^2_{\V}.
\end{align*}
From \eqref{conv_z2} and the fact that $\v^0\in\mathrm{L}^2_{\mathrm{loc}}(\tau,+\infty;\V)$ imply that
\begin{align}\label{LC15}
	\int_{\tau}^{\tau+T}P_1(t)\d t\leq C(\tau,T,\omega_0).
\end{align}
Now, from the fact that $\f\in\mathrm{L}^2_{\text{loc}}(\R;\H)$, $\v^0\in\mathrm{C}([\tau,+\infty);\H)\cap\mathrm{L}^2_{\mathrm{loc}}(\tau,+\infty;\V)$ and Lemma \ref{conv_z}, we conclude that
\begin{align}\label{LC16}
	\lim_{k\to+\infty}\int_{\tau}^{\tau+T}Q_1(t)\d t=0.
\end{align}
Making use of Gronwall's inequality in \eqref{LC13}, we get for all $t\in[\tau,\tau+T]$,
\begin{align}\label{LC17}
	\|\mathscr{V}^k(t)\|^2_{\H}\leq e^{\int_{\tau}^{\tau+T}P_1(t)\d t}\left[\int_{\tau}^{\tau+T}Q_1(t)\d t\right].
\end{align}
In view of \eqref{LC15}-\eqref{LC17}, we find for all $t\in[\tau,\tau+T]$,
\begin{align}\label{LC18}
	\|\mathscr{V}^k(t)\|^2_{\H}\to0 \ \text{ as } \ k \to +\infty,
\end{align}
which completes the proof. Furthermore, integrating \eqref{LC13} from $\tau$ to $\tau+T$ and using \eqref{LC15}-\eqref{LC17}, we get
\begin{align}\label{LC19}
	\int_{\tau}^{\tau+T}\|\mathscr{V}^k(t)\|^2_{\V}\d t\to0 \ \text{ as } \ k\to +\infty,
\end{align}
for any $\tau\in\mathbb{R}$.
\end{proof}

Lemma \ref{Soln} ensures us that we can define a mapping $\Phi:\R^+\times\R\times\Omega\times\H\to\H$ by
\begin{align}\label{Phi}
	\Phi(t,\tau,\omega,\u_{\tau})=\u(t+\tau,\tau,\theta_{-\tau}\omega,\u_{\tau})=e^{z(\theta_{t}\omega)}\v(t+\tau,\tau,\theta_{-\tau}\omega,\v_{\tau}).
\end{align}

The Lusin continuity in Proposition \ref{LusinC} provides the $\mathscr{F}$-measurability of $\Phi$. Consequently, in view of Lemma \ref{Soln} and Proposition \ref{LusinC}, we have the following result for NRDS.
\begin{proposition}\label{NRDS}
	The mapping $\Phi$ defined by \eqref{Phi} is an  NRDS on $\H$, that is, $\Phi$ has  the following properties:
	\begin{itemize}
		\item [(i)]$\Phi$ is $(\mathscr{B}(\R^+)\times\mathscr{B}(\R)\times\mathscr{F}\times\mathscr{B}(\H);\mathscr{B}(\H))$ measurable,
		\item [(ii)] $\Phi$ satisfies the cocycle property: $\Phi(0,\tau,\omega,\cdot)=\I$, and
		\begin{align*}
			\Phi(t+s,\tau,\omega,\u_{\tau})=\Phi(t,\tau+s,\theta_s\omega,\Phi(s,\tau,\omega,\u_{\tau})), \ \ \ t,s\geq0.
		\end{align*}
	\end{itemize}
\end{proposition}
\subsubsection{Backward convergence of NRDS}
In this subsection, we first consider the autonomous system corresponding to the non-autonomous system \eqref{SNSE} with $S(\u)=\u$ and prove that the solution to the system \eqref{CNSE-M} converges to the solution of the corresponding autonomous system as $\tau\to-\infty$ in $\H$. Consider the autonomous SNSE with the multiplicative white noise:
\begin{equation}\label{A-SNSE-M}
	\left\{
	\begin{aligned}
		\frac{\d\widetilde{\u}(t)}{\d t}+\nu \A\widetilde{\u}(t)+\B(\widetilde{\u}(t))&=\f_{\infty} +\widetilde{\u}(t)\circ\frac{\d \W(t)}{\d t} , \\
		\widetilde{\u}(x,0)&=\widetilde{\u}_{0}(x),	\ x\in \mathcal{O}.
	\end{aligned}
	\right.
\end{equation}
Let $\widetilde{\v}(t,\omega)=e^{-z(\theta_{t}\omega)}\widetilde{\u}(t,\omega)$. Then, the system \eqref{A-SNSE-M} can be written in the following pathwise deterministic system:
\begin{equation}\label{A-CNSE-M}
	\left\{
	\begin{aligned}
		\frac{\d\widetilde{\v}(t)}{\d t}+\nu \A\widetilde{\v}(t)+e^{z(\theta_{t}\omega)}\B\big(\widetilde{\v}(t)\big)&=\f_{\infty} e^{-z(\theta_{t}\omega)} + \sigma z(\theta_t\omega)\widetilde{\v}(t) , \quad t> \tau, \ \tau\in\R ,\\
		\widetilde{\v}(x,0)&=\widetilde{\v}_{0}(x)=e^{-z(\omega)}\widetilde{\u}_{0}(x), \hspace{13mm} x\in\mathcal{O},
	\end{aligned}
	\right.
\end{equation}
in $\V'$ (in weak sense).
\begin{proposition}\label{Back_conver}
	Suppose that Hypothesis \ref{Hypo_f-N} is satisfied. Then, the solution $\v$ of the system \eqref{CNSE-M} backward converges to the solution $\widetilde{\v}$ of the system \eqref{A-CNSE-M}, that is,
	\begin{align}\label{BC}
		\lim_{\tau\to -\infty}\|\v(T+\tau,\tau,\theta_{-\tau}\omega,\v_{\tau})-\widetilde{\v}(t,\omega,\widetilde{\v}_0)\|_{\H}=0, \ \ \text{ for all } T>0 \text{ and } \omega\in\Omega,
	\end{align}
whenever $\|\v_{\tau}-\widetilde{\v}_0\|_{\H}\to0$ as $\tau\to-\infty.$
\end{proposition}
\begin{proof}
	Let $\mathscr{V}^{\tau}(t):=\v(t+\tau,\tau,\theta_{-\tau}\omega,\v_{\tau})-\widetilde{\v}(t,\omega,\widetilde{\v}_0)$ for $t\geq0$. From \eqref{CNSE-M} and \eqref{A-CNSE-M}, we obtain
	\begin{align}\label{BC1}
		\frac{\d\mathscr{V}^{\tau}}{\d t}&=-\nu \A\mathscr{V}^{\tau}-e^{z(\theta_{t}\omega)}\left[\B\big(\v\big)-\B\big(\widetilde{\v}\big)\right] +e^{-z(\theta_{t}\omega)}\left[\f(t+\tau)-\f_{\infty}\right] + \sigma z(\theta_t\omega)\mathscr{V}^{\tau},
	\end{align}
	in $\V'$ (in weak sense).  Taking the inner product with $\mathscr{V}^{\tau}(\cdot)$ in \eqref{BC1}, and using \eqref{b0} and \eqref{441}, we get
	\begin{align}\label{BC2}
		&\frac{1}{2}\frac{\d }{\d t}\|\mathscr{V}^{\tau}\|^2_{\H}\nonumber\\&=-\nu\|\mathscr{V}^{\tau}\|^2_{\V}+\sigma z(\theta_t\omega)\|\mathscr{V}^{\tau}\|^2_{\H}-e^{z(\theta_{t}\omega)}b(\mathscr{V}^{\tau},\widetilde{\v},\mathscr{V}^{\tau})+e^{-z(\theta_{t}\omega)}(\f(t+\tau)-\f_{\infty},\mathscr{V}^{\tau}).
	\end{align}
Applying H\"older's and Young's inequalities, we infer
\begin{align}\label{BC4}
	\left|e^{-z(\theta_{t}\omega)}(\f(t+\tau)-\f_{\infty},\mathscr{V}^{\tau})\right|\leq\|\f(t+\tau)-\f_{\infty}\|^2_{\H}+Ce^{-2z(\theta_{t}\omega)}\|\mathscr{V}^{\tau}\|^2_{\H},
\end{align}
and
\begin{align}\label{BC5}
	&\left|e^{z(\theta_{t}\omega)}b(\mathscr{V}^{\tau},\widetilde{\v},\mathscr{V}^{\tau})\right|\leq Ce^{2z(\theta_{t}\omega)}\|\widetilde{\v}\|^2_{\V}\|\mathscr{V}^{\tau}\|^2_{\H}+\frac{\nu}{2}\|\mathscr{V}^{\tau}\|^2_{\V},
\end{align}
where we have used \eqref{b1} also in \eqref{BC5}. Combining \eqref{BC2}-\eqref{BC5}, we arrive at
\begin{align}\label{BC6}
	\frac{\d }{\d t}\|\mathscr{V}^{\tau}\|^2_{\H}+\frac{\nu}{2}\|\mathscr{V}^{\tau}\|^2_{\V}\leq C\big[
		S_1(t)\|\mathscr{V}^{\tau}\|^2_{\H}+\|\f(t+\tau)-\f_{\infty}\|^2_{\H}\big],
\end{align}
where
\begin{align*}
	S_1(t)=e^{2z(\theta_{t}\omega)}\|\widetilde{\v}\|^2_{\V}+e^{-2z(\theta_{t}\omega)}+\left|z(\theta_{t}\omega)\right|.
\end{align*}
Making use of Gronwall's inequality to \eqref{BC6} over $(0,T)$, we obtain
\begin{align}\label{BC7}
	\|\mathscr{V}^{\tau}(T)\|^2_{\H}\leq \left[\|\mathscr{V}^{\tau}(0)\|^2_{\H}+C\int_{0}^{T}\|\f(t+\tau)-\f_{\infty}\|^2_{\H} \d t\right]e^{C\int_{0}^{T}S_1(t)\d t}.
\end{align}
Since $z$ is continuous and $\widetilde{\v}\in\mathrm{L}^2(0,T;\V)$, it implies that
\begin{align}\label{BC8}
	\int_{0}^{T}S_{1}(t)\d t<+\infty.
\end{align}
 From Hypothesis \ref{Hypo_f-N}, we deduce that
\begin{align}\label{BC9}
	\int_{0}^{T}\|\f(t+\tau)-\f_{\infty}\|^2_{\H} \d t\leq \int_{-\infty}^{\tau+T}\|\f(t)-\f_{\infty}\|^2_{\H} \d t\to 0 \ \text{ as } \ \tau\to -\infty.
\end{align}
Using the fact that $\int_{0}^{T}S_{1}(t)\d t$ is bounded, \eqref{BC9} and $\|\mathscr{V}^{\tau}(0)\|^2_{\H}=\|\v_{\tau}-\widetilde{\v}_0\|_{\H}\to0$ as $\tau\to-\infty$, one can conclude the proof.  Furthermore, integrating \eqref{BC6} from $0$ to $T$ and using \eqref{BC}, \eqref{BC7}-\eqref{BC9}, we get
\begin{align}\label{BC10}
	\int_{0}^{T}\|\mathscr{V}^{\tau}(t)\|^2_{\V}\d t \to 0 \ \text{ as }\  \tau \to -\infty,
\end{align}
for all $T>0$.
\end{proof}
\subsubsection{Increasing random absorbing set}
This subsection provides the existence of increasing $\mathfrak{D}$-random absorbing set in $\H$ for the non-autonomous SNSE \eqref{SNSE} with $S(\u)=\u$.
\begin{lemma}\label{Absorbing}
	Suppose that $\f\in\mathrm{L}^2_{\mathrm{loc}}(\R;\H)$. Then, for each $(\tau,\omega,D)\in\R\times\Omega\times\mathfrak{D},$ there exists a time $\mathcal{T}:=\mathcal{T}(\tau,\omega,D)\geq2$ such that
	\begin{align}\label{AB1}
		&\sup_{s\leq \tau}\sup_{t\geq \mathcal{T}}\sup_{\v_{0}\in D(s-t,\theta_{-t}\omega)}\|\v(s,s-t,\theta_{-s}\omega,\v_{0})\|^2_{\H}\leq 1+\frac{2}{\nu\lambda_{1}}K(\tau,\omega),
	\end{align}
\begin{align}\label{AB11}
	\sup_{s\leq \tau}\sup_{t\geq \mathcal{T}}\sup_{\v_{0}\in D(s-t,\theta_{-t}\omega)}\sup_{\ell\in[s-2,s]}\|\v(\ell,s-t,\theta_{-s}\omega,\v_{0})\|^2_{\H}\leq e^{2\nu\lambda_{1}}\bigg[ 1+\frac{2}{\nu\lambda_{1}}K(\tau,\omega)\bigg],
\end{align}
and
\begin{align}\label{AB111}
	&\sup_{s\leq \tau}\sup_{t\geq \mathcal{T}}\sup_{\v_{0}\in D(s-t,\theta_{-t}\omega)}\int_{s-2}^{s}\|\v(\ell,s-t,\theta_{-s}\omega,\v_{0})\|^2_{\V}\d \ell\leq\widetilde{K}(\tau,\omega),
\end{align}
where $K(\tau,\omega)$ and $\widetilde{K}(\tau,\omega)$ are given by
\begin{align}\label{AB2}
	K(\tau,\omega):=\sup_{s\leq \tau}\int_{-\infty}^{0}e^{\nu\lambda_{1}\uprho+2|z(\theta_{\uprho}\omega)|+2\sigma\int_{\uprho}^{0}z(\theta_{\upeta}\omega)\d\upeta}\|\f(\uprho+s)\|^2_{\H}\d\uprho,
\end{align}
and
\begin{align}\label{AB22}
\widetilde{K}(\tau,\omega)&:= e^{2\nu\lambda_{1}}	\bigg[1+\frac{2}{\nu\lambda_{1}}K(\tau,\omega)\bigg]\bigg[1+2\sigma\int_{-2}^0\left|z(\theta_{\uprho}\omega)\right|\d\uprho\bigg]\nonumber\\&\quad+\frac{2}{\nu\lambda_{1}}\sup_{s\leq \tau}\int_{-2}^{0}e^{2|z(\theta_{\uprho}\omega)|}\|\f(\uprho+s)\|^2_{\H}\d\uprho,
\end{align}
respectively.
\end{lemma}
\begin{proof}
		From the first equation of the system \eqref{CNSE-M} and \eqref{b0}, we obtain
	
	\begin{align*}
		\frac{1}{2}\frac{\d}{\d t} \|\v\|^2_{\H} +\frac{3\nu}{4}\|\v\|^2_{\V}+\frac{\nu}{4}\|\v\|^2_{\V}&= e^{-z(\theta_{t}\omega)}\left(\f,\v\right)+\sigma z(\theta_{t}\omega)\|\v\|^2_{\H}\nonumber\\&\leq \frac{\nu\lambda_{1}}{4}\|\v\|^2_{\H}+\frac{e^{2\left|z(\theta_{t}\omega)\right|}}{\nu\lambda_{1}}\|\f\|^2_{\H}+\sigma z(\theta_{t}\omega)\|\v\|^2_{\H}.
	\end{align*}	
	Now, using \eqref{poin} in the second term on the left hand side of above inequality, we get
		\begin{align}\label{EI1}
		\frac{\d}{\d t} \|\v\|^2_{\H}+ \left(\nu\lambda_{1}-2\sigma z(\theta_{t}\omega)\right)\|\v\|^2_{\H}+\frac{\nu}{2}\|\v\|^2_{\V} \leq \frac{2e^{2\left|z(\theta_{t}\omega)\right|}}{\nu\lambda_{1}}\|\f\|^2_{\H}.
	\end{align}
	 Let us rewrite the energy inequality \eqref{EI1} for $\v(\zeta)=\v(\zeta,s-t,\theta_{-s}\omega,\v_{0})$, that is,
	\begin{align}\label{AB0}
		\frac{\d}{\d\zeta} \|\v(\zeta)\|^2_{\H}&+ \left(\nu\lambda_{1}-2\sigma z(\theta_{\zeta-s}\omega)\right)\|\v(\zeta)\|^2_{\H}+\frac{\nu}{2}\|\v(\zeta)\|^2_{\V}\leq \frac{2e^{2|z(\theta_{\zeta-s}\omega)|}}{\nu\lambda_{1}}\|\f(\zeta)\|^2_{\H}.
	\end{align}
In view of variation of constants formula with respect to $\zeta\in(s-t,\xi)$, we get
\begin{align}\label{AB3}
	&\|\v(\xi,s-t,\theta_{-s}\omega,\v_{0})\|^2_{\H}+\frac{\nu}{2}\int_{s-t}^{\xi}e^{\nu\lambda_{1}(\uprho-\xi)-2\sigma\int^{\uprho}_{\xi}z(\theta_{\upeta-s}\omega)\d\upeta}\|\v(\uprho,s-t,\theta_{-s}\omega,\v_{0})\|^2_{\V}\d\uprho\nonumber\\&\leq e^{-\nu\lambda_{1}(\xi-s+t)+2\sigma\int_{-t}^{\xi-s}z(\theta_{\upeta}\omega)\d\upeta}\|\v_{0}\|^2_{\H} + \frac{2}{\nu\lambda_{1}}\int\limits_{-t}^{\xi-s}e^{\nu\lambda_{1}(\uprho+s-\xi)+2|z(\theta_{\uprho}\omega)|+2\sigma\int_{\uprho}^{\xi-s}z(\theta_{\upeta}\omega)\d\upeta}\|\f(\uprho+s)\|^2_{\H}\d\uprho.
\end{align}
Putting $\xi=s$ in \eqref{AB3}, we find
\begin{align}\label{AB4}
	&\|\v(s,s-t,\theta_{-s}\omega,\v_{0})\|^2_{\H}+\frac{\nu}{2}\int_{s-t}^{s}e^{\nu\lambda_{1}(\uprho-s)-2\sigma\int^{\uprho}_{s}z(\theta_{\upeta-s}\omega)\d\upeta}\|\v(\uprho,s-t,\theta_{-s}\omega,\v_{0})\|^2_{\V}\d\uprho\nonumber\\&\leq e^{-\nu\lambda_{1} t+2\sigma\int_{-t}^{0}z(\theta_{\upeta}\omega)\d\upeta}\|\v_{0}\|^2_{\H} +\frac{2}{\nu\lambda_{1}}\int_{-t}^{0}e^{\nu\lambda_{1}\uprho+2|z(\theta_{\uprho}\omega)|+2\sigma\int_{\uprho}^{0}z(\theta_{\upeta}\omega)\d\upeta}\|\f(\uprho+s)\|^2_{\H}\d\uprho\nonumber\\&\leq e^{-\nu\lambda_{1} t+2\sigma\int_{-t}^{0}z(\theta_{\upeta}\omega)\d\upeta}\|\v_{0}\|^2_{\H} +\frac{2}{\nu\lambda_{1}}\int_{-\infty}^{0}e^{\nu\lambda_{1}\uprho+2|z(\theta_{\uprho}\omega)|+2\sigma\int_{\uprho}^{0}z(\theta_{\upeta}\omega)\d\upeta}\|\f(\uprho+s)\|^2_{\H}\d\uprho,
\end{align}
for all $s\leq\tau$. Since $\v_0\in D(s-t,\theta_{-t}\omega)$ and $D$ is backward tempered, it implies from \eqref{Z4} and the definition of backward temperedness \eqref{BackTem} that there exists a time $\mathcal{T}=\mathcal{T}(\tau,\omega,D)$ such that for all $t\geq \mathcal{T}$,
\begin{align}\label{v_0}
	e^{-\nu\lambda_{1} t+2\sigma\int_{-t}^{0}z(\theta_{\upeta}\omega)\d\upeta}\sup_{s\leq \tau}\|\v_{0}\|^2_{\H}\leq e^{-\frac{\nu\lambda_{1}}{3}t}\sup_{s\leq \tau}\|D(s-t,\theta_{-t}\omega)\|^2_{\H}\leq1.
\end{align}
Hence, by taking supremum on $s\in(-\infty,\tau]$ in \eqref{AB4},  we achieve at \eqref{AB1}.

 Furthermore, in view of \eqref{AB3} and \eqref{AB22}, we have for $t\geq\mathcal{T}$,
 \begin{align}\label{AB5}
 	&\sup_{s\leq \tau}\sup_{t\geq \mathcal{T}}\sup_{\v_{0}\in D(s-t,\theta_{-t}\omega)}\sup_{l\in[s-2,s]}\|\v(l,s-t,\theta_{-s}\omega,\v_{0})\|^2_{\H}\nonumber\\&\leq\sup_{l\in[s-2,s]}\left[e^{-\nu\lambda_{1}(l-s)}\right]\bigg[ 1+\frac{2}{\nu\lambda_{1}}K(\tau,\omega)\bigg]=e^{2\nu\lambda_{1}}\bigg[ 1+\frac{2}{\nu\lambda_{1}}K(\tau,\omega)\bigg],
 \end{align}
as required in \eqref{AB11}.

Now, integrating \eqref{AB0} over $(s-2,s)$ and taking supremum on $s\in(-\infty,\tau]$, we reach at
\begin{align}\label{AB6}
	&\sup_{s\leq \tau}\int_{s-2}^{s}\|\v(\uprho,s-t,\theta_{-s}\omega,\v_{0})\|^2_{\V}\d\uprho\nonumber\\&\leq \sup_{s\leq \tau}\|\v(s-2,s-t,\theta_{-s}\omega,\v_{0})\|^2_{\H} + 2\sigma\sup_{s\leq \tau}\int_{s-2}^{s}\left|z(\theta_{\uprho-s}\omega)\right|\|\v(\uprho,s-t,\theta_{-s}\omega,\v_{0})\|^2_{\H}\d\uprho\nonumber\\&\quad+\frac{2}{\nu\lambda_{1}}\sup_{s\leq \tau}\int_{s-2}^{s}e^{2|z(\theta_{\uprho-s}\omega)|}\|\f(\uprho)\|^2_{\H}\d\uprho\nonumber\\&\leq \sup_{s\leq \tau}\|\v(s-2,s-t,\theta_{-s}\omega,\v_{0})\|^2_{\H} + 2\sigma\sup_{s\leq \tau}\sup_{\uprho\in[s-2,s]}\|\v(\uprho,s-t,\theta_{-s}\omega,\v_{0})\|^2_{\H}\int_{-2}^{0}\left|z(\theta_{\uprho}\omega)\right|\d\uprho\nonumber\\&\quad+\frac{2}{\nu\lambda_{1}}\sup_{s\leq \tau}\int_{-2}^{0}e^{2|z(\theta_{\uprho}\omega)|}\|\f(\uprho+s)\|^2_{\H}\d\uprho.
\end{align}
Using \eqref{AB1} and \eqref{AB11} in \eqref{AB6}, we obtain \eqref{AB111}, which completes the proof
\end{proof}

\begin{lemma}\label{AbsorbingV}
	Suppose that $\f\in\mathrm{L}^2_{\mathrm{loc}}(\R;\H)$. Then, for each $(\tau,\omega,D)\in\R\times\Omega\times\mathfrak{D},$ there exists a time $\mathcal{T}:=\mathcal{T}(\tau,\omega,D)\geq2$ $($same as in Lemma \ref{Absorbing}$)$ such that for $\xi\in[s-1,s]$,
	\begin{align}\label{AB1-V}
		&\sup_{s\leq \tau}\sup_{t\geq \mathcal{T}}\sup_{\v_{0}\in D(s-t,\theta_{-t}\omega)}\|\v(\xi,s-t,\theta_{-s}\omega,\v_{0})\|^2_{\V}\nonumber\\&\leq \bigg[\widetilde{K}(\tau,\omega)+\frac{2}{\nu}\sup_{\zeta\in[-2,0]}\left[e^{2|\z(\theta_{\zeta}\omega)|}\right]\sup_{s\leq\tau}\int_{-2}^{0}\|\f(\zeta+s)\|^2_{\H}\d\zeta\bigg]e^{K_1(\tau,\omega)}=:\widehat{K}(\tau,\omega),
	\end{align}
and
\begin{align}\label{AB1-D(A)}
&\sup_{s\leq \tau}\sup_{t\geq \mathcal{T}}\sup_{\v_{0}\in D(s-t,\theta_{-t}\omega)}	\int_{s-1}^{s}\|\A\v(\zeta,s-t,\theta_{-s}\omega,\v_{0})\|^{2}_{\H}\d\zeta\nonumber\\&\leq \frac{\widehat{K}(\tau,\omega)}{\nu}+\frac{1}{\nu}\bigg[2\sigma\sup_{\zeta\in[-1,0]}\left\{\z(\theta_{\zeta}\omega)\right\}+Ce^{2\nu\lambda_{1}}\left\{\widehat{K}(\tau,\omega)\right\}^3\sup_{\zeta\in[-1,0]}\left\{e^{4|\z(\theta_{\zeta}\omega)|}\right\}\bigg]\nonumber\\&\quad+\frac{2}{\nu^2}\sup_{\zeta\in[-1,0]}\left[e^{2|\z(\theta_{\zeta}\omega)|}\right]\int_{s-1}^{s}\|\f(\zeta)\|^2_{\H}\d\zeta=:\widehat{K}_1(\tau,\omega),
\end{align}
where
\begin{align*}
	K_1(\tau,\omega):=2\sigma\sup_{\zeta\in[-2,0]}\left|\z(\theta_{\zeta}\omega)\right|+C\widetilde{K}(\tau,\omega)\sup_{\zeta\in[-2,0]}\left[e^{4|\z(\theta_{\zeta}\omega)|}\right]\bigg[1+\frac{2}{\nu\lambda_{1}}K(\tau,\omega)\bigg],
\end{align*}
and, $K(\tau,\omega)$ and $\widetilde{K}(\tau,\omega)$ are given by \eqref{AB2} and \eqref{AB22}, respectively.
\end{lemma}
\begin{proof}
	Taking the inner product of the first equation in \eqref{CNSE-M} with $\A\v(\cdot)$, we find
	\begin{align}\label{BddV1}
		&\frac{1}{2}\frac{\d}{\d\zeta}\|\v(\zeta)\|^2_{\V}+\nu\|\A\v(\zeta)\|^2_{\H}-\sigma\z(\theta_{\zeta}\omega)\|\v\|^2_{\V}+e^{\z(\theta_{\zeta}\omega)}b(\v(\zeta),\v(\zeta),\A\v(\zeta)) \nonumber\\&= e^{-\z(\theta_{\zeta}\omega)}(\f(\zeta),\A\v(\zeta))\leq e^{-\z(\theta_{\zeta}\omega)}\|\f(\zeta)\|_{\H}\|\A\v(\zeta)\|_{\H}\nonumber\\&\leq\frac{e^{2|\z(\theta_{\zeta}\omega)|}}{\nu}\|\f(\zeta)\|^2_{\H}+\frac{\nu}{4}\|\A\v(\zeta)\|^2_{\H}.
	\end{align}
	By \eqref{b2} and Young's inequality, we have
	\begin{align}\label{BddV2}
		\left|e^{\z(\theta_{\zeta}\omega)}b(\v,\v,\A\v)\right|&\leq Ce^{|\z(\theta_{\zeta}\omega)|}\|\v\|^{1/2}_{\H}\|\v\|_{\V}\|\A\v\|^{3/2}_{\H}\nonumber\\&\leq Ce^{4|\z(\theta_{\zeta}\omega)|}\|\v\|^{2}_{\H}\|\v\|^{4}_{\V}+\frac{\nu}{4}\|\A\v\|^{2}_{\H}.
	\end{align}
	From \eqref{BddV1}-\eqref{BddV2}, we obtain
	\begin{align}\label{BddV2'}
		&\frac{\d}{\d\zeta}\|\v(\zeta)\|^2_{\V}+\nu\|\A\v(\zeta)\|^2_{\H}\nonumber\\&\leq2\sigma\z(\theta_{\zeta}\omega)\|\v(\zeta)\|^2_{\V}+\frac{2e^{2|\z(\theta_{\zeta}\omega)|}}{\nu}\|\f(\zeta)\|^2_{\H}+Ce^{4|\z(\theta_{\zeta}\omega)|}\|\v(\zeta)\|^{2}_{\H}\|\v(\zeta)\|^{4}_{\V}.
	\end{align}
Replacing $\omega$ with $\theta_{-s}\omega$ in the above inequality, we find
	\begin{align}\label{BddV3}
	&\frac{\d}{\d\zeta}\|\v(\zeta,s-t,\theta_{-s}\omega,\v_{0})\|^2_{\V}\nonumber\\&\leq\bigg[2\sigma\z(\theta_{\zeta-s}\omega)+Ce^{4|\z(\theta_{\zeta-s}\omega)|}\|\v(\zeta,s-t,\theta_{-s}\omega,\v_{0})\|^{2}_{\H}\|\v(\zeta,s-t,\theta_{-s}\omega,\v_{0})\|^{2}_{\V}\bigg]\nonumber\\&\quad\times\|\v(\zeta,s-t,\theta_{-s}\omega,\v_{0})\|^2_{\V}+\frac{2e^{2|\z(\theta_{\zeta-s}\omega)|}}{\nu}\|\f(\zeta)\|^2_{\H}.
\end{align}
Using \eqref{AB11} and \eqref{AB111}, for $t\geq \mathcal{T}$ and $\v_{0}\in D(s-t,\theta_{-t}\omega)$, we obtain for $\xi\in[s-1,s]$,
\begin{align}
	&\sup_{s\leq\tau}\int_{s-2}^{\xi}\|\v(\zeta,s-t,\theta_{-s}\omega,\v_{0})\|^2_{\V}\d\zeta\leq\sup_{s\leq\tau}\int_{s-2}^{s}\|\v(\zeta,s-t,\theta_{-s}\omega,\v_{0})\|^2_{\V}\d\zeta\leq\widetilde{K}(\tau,\omega),\label{BddV4}\\
	\nonumber\\
	&\sup_{s\leq\tau}\int_{s-2}^{\xi}\bigg[2\sigma\z(\theta_{\zeta-s}\omega)+Ce^{4|\z(\theta_{\zeta-s}\omega)|}\|\v(\zeta,s-t,\theta_{-s}\omega,\v_{0})\|^{2}_{\H}\|\v(\zeta,s-t,\theta_{-s}\omega,\v_{0})\|^{2}_{\V}\bigg]\d\zeta\nonumber\\&\leq\sup_{s\leq\tau}\int_{s-2}^{s}\bigg[2\sigma\z(\theta_{\zeta-s}\omega)\nonumber\\&\qquad\qquad\qquad+Ce^{4|\z(\theta_{\zeta-s}\omega)|}\|\v(\zeta,s-t,\theta_{-s}\omega,\v_{0})\|^{2}_{\H}\|\v(\zeta,s-t,\theta_{-s}\omega,\v_{0})\|^{2}_{\V}\bigg]\d\zeta\nonumber\\&\leq2\sigma\sup_{\zeta\in[-2,0]}\left|\z(\theta_{\zeta}\omega)\right|\nonumber\\&\quad+C\sup_{\zeta\in[-2,0]}\left[e^{4|\z(\theta_{\zeta}\omega)|}\right]\sup_{\zeta\in[s-2,s]}\left[\|\v(\zeta,s-t,\theta_{-s}\omega,\v_{0})\|^{2}_{\H}\right]\int_{s-2}^{s}\|\v(\zeta,s-t,\theta_{-s}\omega,\v_{0})\|^2_{\V}\d\zeta\nonumber\\&\leq 2\sigma\sup_{\zeta\in[-2,0]}\left|\z(\theta_{\zeta}\omega)\right|+C\widetilde{K}(\tau,\omega)\sup_{\zeta\in[-2,0]}\left[e^{4|\z(\theta_{\zeta}\omega)|}\right]\bigg[1+\frac{2}{\nu\lambda_{1}}K(\tau,\omega)\bigg]:=K_1(\tau,\omega)<+\infty,\label{BddV5}
\end{align}
and
\begin{align}
	\sup_{s\leq\tau}\int_{s-2}^{\xi}\frac{2e^{2|\z(\theta_{\zeta-s}\omega)|}}{\nu}\|\f(\zeta)\|^2_{\H}\d\zeta&\leq\sup_{s\leq\tau}\int_{s-2}^{s}\frac{2e^{2|\z(\theta_{\zeta-s}\omega)|}}{\nu}\|\f(\zeta)\|^2_{\H}\d\zeta \nonumber\\&\leq\frac{2}{\nu}\sup_{\zeta\in[-2,0]}\left[e^{2|\z(\theta_{\zeta}\omega)|}\right]\sup_{s\leq\tau}\int_{-2}^{0}\|\f(\zeta+s)\|^2_{\H}\d\zeta.\label{BddV6}
\end{align}
Consider,
\begin{align*}
	\sup_{s\leq\tau}\int_{-2}^{0}\|\f(\zeta+s)\|^2_{\H}\d\zeta&\leq e^{2\lambda_{1}}\sup_{s\leq\tau}\int_{-2}^{0}e^{\lambda_{1}\zeta}\|\f(\zeta+s)\|^2_{\H}\d\zeta\nonumber\\&\leq e^{2\lambda_{1}}\sup_{s\leq\tau}\int_{-\infty}^{0}e^{\lambda_{1}\zeta}\|\f(\zeta+s)\|^2_{\H}\d\zeta<+\infty,
\end{align*}
where we have used \eqref{f2-N}. Hence, in view of \eqref{BddV4}-\eqref{BddV6}, we apply the uniform Gronwall lemma (Lemma 1.1, \cite{R.Temam}) to the inequality \eqref{BddV3} to obtain \eqref{AB1-V}.

Further, integrating \eqref{BddV2'} from $s-1$ to $s$, we get
\begin{align*}
	&\int_{s-1}^{s}\|\A\v(\zeta,s-t,\theta_{-s}\omega,\v_{0})\|^{2}_{\H}\d\zeta\nonumber\\&\leq\frac{1}{\nu}\|\v(s-1,s-t,\theta_{-s}\omega,\v_{0})\|^2_{\V}\nonumber\\&\quad+\frac{1}{\nu}\bigg[2\sigma\sup_{\zeta\in[-1,0]}\left[\z(\theta_{\zeta}\omega)\right]+C\sup_{\zeta\in[-1,0]}\left[e^{4|\z(\theta_{\zeta}\omega)|}\right]\sup_{\zeta\in[s-1,s]}\|\v(\zeta,s-t,\theta_{-s}\omega,\v_{0})\|^{6}_{\V}\bigg]\nonumber\\&\quad+\frac{2}{\nu^2}\sup_{\zeta\in[-1,0]}\left[e^{2|\z(\theta_{\zeta}\omega)|}\right]\int_{s-1}^{s}\|\f(\zeta)\|^2_{\H}\d\zeta.
\end{align*}
Hence, using \eqref{AB1-V}, we arrive at \eqref{AB1-D(A)}, which completes the proof.
\end{proof}

\begin{proposition}\label{IRAS}
	Suppose that $\f\in\mathrm{L}^2_{\mathrm{loc}}(\R;\H)$. Then, there is an increasing random absorbing set $\mathcal{K}_{\H}\subset\H$ given by
	\begin{align}\label{IRAS1}
		\mathcal{K}_{\H}(\tau,\omega):=\left\{\u\in\H:\|\u\|^2_{\H}\leq e^{z(\omega)}\left[1+\frac{2}{\nu\lambda_{1}}K(\tau,\omega)\right]\right\}, \ \text{ for all } \ \tau\in\R,
	\end{align}
	where $K(\tau,\omega)$ is defined by \eqref{AB2}. Moreover, $\mathcal{K}_{\H}$ is backward tempered, that is, $\mathcal{K}_{\H}\in\mathfrak{D}$.
\end{proposition}
\begin{proof}
	Using \eqref{Z3}, \eqref{Z4} and \eqref{f2-N}, we obtain that
	\begin{align}\label{Absorbing-R}
		K(\tau,\omega)\leq C\sup_{s\leq \tau}\int_{-\infty}^{0} e^{\frac{\nu\lambda_{1}}{2}\uprho}\|\f(\uprho+s)\|^2_{\H}\d\uprho<+\infty,
	\end{align}
	and one can deduce that $\mathcal{K}_{\H}$ is tempered. Further, $\mathcal{K}_{\H}(\tau,\omega)$ is increasing due to the fact that $\tau\mapsto K(\tau,\omega)$ is an increasing function. Then, $\mathcal{K}_{\H}$ is an increasing tempered set which gives the backward temperedness of $\mathcal{K}_{\H}(\tau,\omega)$, that is, $\mathcal{K}_{\H}\in\mathfrak{D}$. We have from Lemma \ref{Absorbing} that for each $\tau\in\R$, $\omega\in\Omega$ and $D\in\mathfrak{D}$, there exists $\mathcal{T}=\mathcal{T}(\tau,\omega,D)\geq2$ such that for all $t\geq\mathcal{T}$,
	\begin{align*}
		\Phi(t,\tau-t,\theta_{-t}\omega,D(\tau-t,\theta_{-t}\omega))=\u(\tau,\tau-t,\theta_{-\tau}\omega,D(\tau-t,\theta_{-t}\omega))\subseteq\mathcal{K}_{\H}(\tau,\omega),
	\end{align*}
	and  the absorption follows. The measurability of the absorbing set $\mathcal{K}_{\H}$ is not immediate. One has to carefully treat the supremum when $s\in(-\infty,\tau]$, that is, $s$ belongs to an uncountable interval. Since the steps are analogous to  Proposition 3.1, \cite{YR}, we omit it here. It can be proved using Egoroff and Lusin theorems as done  in Proposition 3.1, \cite{YR}.
\end{proof}
	\subsubsection{Backward compact random attractors and their asymptotic autonomy}
	In this subsection, we demonstrate the main result of this section, that is, the existence of a backward compact random attractor and asymptotic autonomy of random attractors in $\H$ for the solution of the system \eqref{SNSE}. For the existence of a unique random attractor for autonomous SNSE driven by multiplicative noise on bounded domains, see \cite{CF}.

\begin{theorem}\label{MT1}
	Suppose that Hypothesis \ref{Hypo_f-N} is satisfied. Then, the non-autonomous random dynamical system $\Phi$ generated by the system \eqref{SNSE} with $S(\u)=\u$ has a backward compact random attractor $\mathcal{A}$ in $\H$ such that $\mathcal{A}$ is a backward tempered set. Furthermore, this backward compact random attractor $\mathcal{A}$ backward converges to $\mathcal{A}_{\infty}$, that is,
	\begin{align}\label{MT2}
		\lim_{\tau\to -\infty}\mathrm{dist}_{\H}(\mathcal{A}(\tau,\omega),\mathcal{A}_{\infty}(\omega))=0\  \emph{ for all } \ \omega\in\Omega,
	\end{align}
	where $\mathcal{A}_{\infty}$ is the random attractor for the autonomous system \eqref{A-CNSE-M}. For any sequence $\tau_{n}\to-\infty$ as $n\to+\infty$, there is a subsequence $\{\tau_{n_k}\}\subseteq\{\tau_{n}\}$  such that
	\begin{align}\label{MT3}
		\lim_{k\to +\infty}\mathrm{dist}_{\H}(\mathcal{A}(\tau_{n_k},\theta_{\tau_{n_k}}\omega),\mathcal{A}_{\infty}(\theta_{\tau_{n_k}}\omega))=0, \  \emph{ for all } \ \omega\in\Omega.
	\end{align}
\end{theorem}

\begin{proof}
	In order to complete the proof, we use an abstract result proved in \cite{YR} (see Theorem \ref{Abstract-result}). In order to do this, we verify  all the assumptions $(a_1)$, $(a_2)$, $(b_1)$ and $(b_2)$ of Theorem \ref{Abstract-result}.
	\vskip 2mm
	\noindent
	\textit{Verification of $(a_1)$ assumption}: It implies from Proposition \ref{IRAS} that NRDS $\Phi$ generated by the system \eqref{SNSE} with $S(\u)=\u$ has  an increasing random absorbing set $\mathcal{K}$ which is backward tempered, that is, $\mathcal{K}\in\mathfrak{D}$. Hence, the condition  $(a_1)$  of Theorem \ref{Abstract-result} is verified.
	\vskip 2mm
	\noindent
	\textit{Verification of $(a_2)$ assumption}: Next, we establish that $\Phi$ is $\mathfrak{D}$-backward asymptotically compact in $\H$.
	
	For this purpose, let us fix $(\tau,\omega,D)\in\R\times\Omega\times\mathfrak{D}$ and take arbitrary sequences $s_n\leq\tau$, $\tau_{n}\to+\infty$ and $\v_{0,n}\in D(s_n-t_n,\theta_{-t_n}\omega)$. We show that the sequence $$\{\v(s_n,s_n-t_n,\theta_{-s_n}\omega,\v_{0,n})\}_{n\in\N}$$ is precompact. Lemma \ref{AbsorbingV} implies that there exists an $\mathcal{N}_0\in\N$ such that $t_n\geq \mathcal{T}(\tau,\omega,D)$ and the sequence $\{\v(s_n,s_n-t_n,\theta_{-s_n}\omega,\v_{0,n})\}_{n\in\mathcal{N}_0}$ is bounded in $\V$. In bounded domains, $\V$ is compactly embedded in $\H$, and it  implies that the sequence $$\{\v(s_n,s_n-t_n,\theta_{-s_n}\omega,\v_{0,n})\}_{n\in\N}$$ has a strongly convergent subsequence in $\H$. Hence, $\Phi$ is $\mathfrak{D}$-backward asymptotically compact in $\H$.
	
	Up to this point, we have shown that the assumptions $(a_1)$ and $(a_2)$ of Theorem \ref{Abstract-result} are satisfied. Therefore, the NRDS $\Phi$ generated by SNSE has a unique backward compact random attractor $\mathcal{A}\in\mathfrak{D}$.
	
	\vskip 2mm
	\noindent
	\textit{Verification of $(b_1)$ assumption}: Proposition \ref{Back_conver} gives the backward convergence of the NRDS $\Phi$ to the RDS $\Phi_{\infty}$, that is, the assumption $(b_2)$ of Theorem \ref{Abstract-result} is verified.
	\vskip 2mm
	\noindent
	\textit{Verification of $(b_2)$ assumption}: Now, it is only remain to verify the assumption $(b_2)$ of Theorem \ref{Abstract-result}, that is, $\mathcal{K}_{-1}\in\mathfrak{D}_{\infty}$, where $\mathcal{K}_{-1}(\omega)=\cup_{\tau\leq-1}\mathcal{K}(\tau,\omega)$. Since, $\mathcal{K}(\tau,\omega)$ is increasing in the $\tau$, which implies that $\mathcal{K}_{-1}(\omega)=\mathcal{K}(-1,\omega)$. From the definition \eqref{IRAS1} of $\mathcal{K}$ , we obtain
	\begin{align*}
		e^{-\frac{\nu\lambda_{1}}{3}t}\|\mathcal{K}_{-1}(\theta_{-t}\omega)\|^2_{\H}\leq e^{-\frac{\nu\lambda_{1}}{3}t}+Ce^{-\frac{\nu\lambda_{1}}{3}t}K(-1,\theta_{-t}\omega).
		\end{align*}
	By \eqref{Z3}, \eqref{Z4} and the backward temperedness of $\f$ (see Lemma \ref{Hypo_f1-N}), we get
	\begin{align}\label{K-T}
		e^{-\frac{\nu\lambda_{1}}{3}t}K(-1,\theta_{-t}\omega)&=e^{-\frac{\nu\lambda_{1}}{3}t}\sup_{s\leq -1}\int_{-\infty}^{0}e^{\nu\lambda_{1}\uprho+2|z(\theta_{\uprho-t}\omega)|+2\sigma\int_{\uprho}^{0}z(\theta_{\upeta-t}\omega)\d\upeta}\|\f(\uprho+s)\|^2_{\H}\d\uprho\nonumber\\&=e^{-\frac{\nu\lambda_{1}}{3}t}\sup_{s\leq -1}\int_{-\infty}^{s}e^{\nu\lambda_{1}(\uprho-s)+2|z(\theta_{\uprho-s-t}\omega)|+2\sigma\int_{\uprho-s-t}^{-t}z(\theta_{\upeta}\omega)\d\upeta}\|\f(\uprho)\|^2_{\H}\d\uprho\nonumber\\&\leq e^{-\frac{\nu\lambda_{1}}{12}t}\sup_{s\leq -1}\int_{-\infty}^{s}e^{\frac{3\nu\lambda_{1}}{4}(\uprho-s)}\|\f(\uprho)\|^2_{\H}\d\uprho\to 0 \text{ as } t\to +\infty,
	\end{align}
which conclude that $\mathcal{K}_{-1}\in\mathfrak{D}_{\infty}$. Hence, the convergences \eqref{MT2} and \eqref{MT3} follow from Theorem \ref{Abstract-result} immediately.
\end{proof}

\begin{remark}
	
\end{remark}

\subsection{Existence and asymptotic autonomy of random attractors in $\V$}
In this section, we prove the existence of backward compact random attractors and their asymptotic autonomy in $\V$. First, we obtain the results which  help us to verify the assumptions of the abstract result established in Theorem \ref{Abstract-result}. Next result shows the Lusin continuity of the mapping with respect to $\omega\in\Omega$ of solution to the system \eqref{CNSE-M}.
\begin{proposition}\label{LusinC-V}
	Suppose that $\f\in\mathrm{L}^2_{\mathrm{loc}}(\R;\H)$. For each $N\in\N$, the mapping $\omega\mapsto\v(t,\tau,\omega,\v_{\tau})$ $($solution of \eqref{CNSE-M}$)$ is continuous from $(\Omega_{N},d_{\Omega_N})$ to $\V$, uniformly in $t\in[\tau,\tau+T]$ with $T>0$.
\end{proposition}
\begin{proof}
	Taking the inner product with $\A\mathscr{V}^k(\cdot)$ in \eqref{LC1}, and using \eqref{b0} and \eqref{b1}, we get
	\begin{align}\label{LC2-V}
		\frac{1}{2}\frac{\d }{\d t}\|\mathscr{V}^k\|^2_{\V}&=-\nu\|\A\mathscr{V}^k\|^2_{\H}+\sigma z(\theta_t\omega_k)\|\mathscr{V}^k\|^2_{\V}-e^{z(\theta_{t}\omega_k)}b(\v^k,\v^k,\A\mathscr{V}^k)\nonumber\\&\quad+e^{z(\theta_{t}\omega_0)}b(\v^0,\v^0,\A\mathscr{V}^k)+\left[e^{-z(\theta_{t}\omega_k)}-e^{-z(\theta_{t}\omega_0)}\right](\f,\A\mathscr{V}^k)\nonumber\\&\quad+\sigma\left[z(\theta_t\omega_k)-z(\theta_t\omega_0)\right](\v^0,\A\mathscr{V}^k)\nonumber\\&=-\nu\|\A\mathscr{V}^k\|^2_{\H}+\sigma z(\theta_t\omega_k)\|\mathscr{V}^k\|^2_{\V}-e^{z(\theta_{t}\omega_k)}b(\mathscr{V}^k,\mathscr{V}^k,\A\mathscr{V}^k)\nonumber\\&\quad-e^{z(\theta_{t}\omega_k)}b(\v^0,\mathscr{V}^k,\A\mathscr{V}^k)-e^{z(\theta_{t}\omega_k)}b(\mathscr{V}^k,\v^0,\A\mathscr{V}^k)\nonumber\\&\quad-\left[e^{-z(\theta_{t}\omega_k)}-e^{-z(\theta_{t}\omega_0)}\right]b(\v^0,\v^0,\A\mathscr{V}^k)+\left[e^{-z(\theta_{t}\omega_k)}-e^{-z(\theta_{t}\omega_0)}\right](\f,\A\mathscr{V}^k)\nonumber\\&\quad+\sigma\left[z(\theta_t\omega_k)-z(\theta_t\omega_0)\right](\v^0,\A\mathscr{V}^k).
	\end{align}
	Using \eqref{poin}, H\"older's and Young's inequalities, we obtain
	\begin{align}
		\left|\left[e^{-z(\theta_{t}\omega_k)}-e^{-z(\theta_{t}\omega_0)}\right](\f,\A\mathscr{V}^k)\right|&\leq C\left|e^{-z(\theta_{t}\omega_k)}-e^{-z(\theta_{t}\omega_0)}\right|^2\|\f\|^2_{\H}+\frac{\nu}{12}\|\A\mathscr{V}^k\|^2_{\H},\label{LC4-V}\\
		\left|\sigma\left[z(\theta_t\omega_k)-z(\theta_t\omega_0)\right](\v^0,\A\mathscr{V}^k)\right|&\leq C\left|z(\theta_t\omega_k)-z(\theta_t\omega_0)\right|^2\|\v^0\|^2_{\H}+\frac{\nu}{12}\|\A\mathscr{V}^k\|^2_{\H}.\label{LC5-V}
	\end{align}
	Using \eqref{b2}, continuous embedding $\D(\A)\subset\V\subset\H$, H\"older's and Young's inequalities, we estimate
	\begin{align}
		\left|e^{z(\theta_{t}\omega_k)}b(\mathscr{V}^k,\mathscr{V}^k,\A\mathscr{V}^k)\right|&\leq Ce^{z(\theta_{t}\omega_k)}\|\mathscr{V}^k\|^{1/2}_{\H}\|\mathscr{V}^k\|_{\V}\|\A\mathscr{V}^k\|^{3/2}_{\H}\nonumber\\&\leq Ce^{4z(\theta_{t}\omega_k)}\|\mathscr{V}^k\|^2_{\H}\|\mathscr{V}^k\|^4_{\V}+\frac{\nu}{12}\|\A\mathscr{V}^k\|^2_{\H},\label{LC7-V}\\
		\left|e^{z(\theta_{t}\omega_k)}b(\v^0,\mathscr{V}^k,\A\mathscr{V}^k)\right|&\leq Ce^{z(\theta_{t}\omega_k)}\|\v^0\|^{1/2}_{\H}\|\v^0\|^{1/2}_{\V}\|\mathscr{V}^k\|^{1/2}_{\V}\|\A\mathscr{V}^k\|^{3/2}_{\H}\nonumber\\&\leq C e^{4z(\theta_{t}\omega_k)}\|\v^0\|^{4}_{\V}\|\mathscr{V}^k\|^{2}_{\V}+\frac{\nu}{12}\|\A\mathscr{V}^k\|^{2}_{\H},\label{LC8-V}\\
		\left|e^{z(\theta_{t}\omega_k)}b(\mathscr{V}^k,\v^0,\A\mathscr{V}^k)\right|&\leq Ce^{z(\theta_{t}\omega_k)}\|\mathscr{V}^k\|^{1/2}_{\H}\|\mathscr{V}^k\|^{1/2}_{\V}\|\v^0\|^{1/2}_{\V}\|\A\v^0\|^{1/2}_{\H}\|\A\mathscr{V}^k\|_{\H}\nonumber\\&\leq C e^{2z(\theta_{t}\omega_k)}\|\A\v^0\|^{2}_{\H}\|\mathscr{V}^k\|^{2}_{\V}+\frac{\nu}{12}\|\A\mathscr{V}^k\|^{2}_{\H},\label{LC9-V}
	\end{align}
	and
	\begin{align}
		\left|\left[e^{z(\theta_{t}\omega_k)}-e^{z(\theta_{t}\omega_0)}\right]b(\v^0,\v^0,\A\mathscr{V}^k)\right|&\leq C\left|e^{z(\theta_{t}\omega_k)}-e^{z(\theta_{t}\omega_0)}\right|\|\v^0\|_{\V}\|\A\v^0\|_{\H}\|\A\mathscr{V}^k\|_{\H}\nonumber\\&\leq C\left|e^{z(\theta_{t}\omega_k)}-e^{z(\theta_{t}\omega_0)}\right|^2\|\v^0\|^2_{\V}\|\A\v^0\|^2_{\H}+\frac{\nu}{12}\|\A\mathscr{V}^k\|^2_{\H}.\label{LC10-V}
	\end{align}
	Combining \eqref{LC2-V}-\eqref{LC8-V}, we arrive at
	\begin{align}\label{LC11-V}
		\frac{\d }{\d t}\|\mathscr{V}^k(t)\|^2_{\V}\leq
		\widehat{P}_1(t)\|\mathscr{V}^k(t)\|^2_{\V}+\widehat{Q}_1(t),
	\end{align}
	for a.e. $t\in[\tau,\tau+T]$, $T>0$, where
	\begin{align*}
		\widehat{P}_1^k&=2\left|z(\theta_{t}\omega_k)\right|+Ce^{4z(\theta_{t}\omega_k)}\|\mathscr{V}^k\|^2_{\H}\|\mathscr{V}^k\|^2_{\V}+Ce^{4z(\theta_{t}\omega_k)}\|\v^0\|^{4}_{\V}+C e^{2z(\theta_{t}\omega_k)}\|\A\v^0\|^{2}_{\H},\\	\widehat{Q}_1^k&=C\left|e^{-z(\theta_{t}\omega_k)}-e^{-z(\theta_{t}\omega_0)}\right|^2\|\f\|^2_{\H}+C\left|z(\theta_t\omega_k)-z(\theta_t\omega_0)\right|^2\|\v^0\|^2_{\H}\nonumber\\&\quad+C\left|e^{z(\theta_{t}\omega_k)}-e^{z(\theta_{t}\omega_0)}\right|^2\|\v^0\|^2_{\V}\|\A\v^0\|^2_{\H}.
	\end{align*}
	Using \eqref{LC18}, \eqref{LC19} and the fact that  $\v^0\in\mathrm{C}([\tau,+\infty);\V)\cap\mathrm{L}^2_{\mathrm{loc}}(\tau,+\infty;\D(\A))$, we obtain
	\begin{align}\label{LC14-V}
		\lim_{k\to+\infty}\int_{\tau}^{\tau+T}\widehat{P}_1^k(t)\d t\leq C(\tau,T,\omega_0).
	\end{align}
	Now, using  the fact that $\f\in\mathrm{L}^2_{\text{loc}}(\R;\H)$, $\v^0\in\mathrm{C}([\tau,+\infty);\V)\cap\mathrm{L}^2_{\mathrm{loc}}(\tau,+\infty;\D(\A))$ and Lemma \ref{conv_z}, we conclude that
	\begin{align}\label{LC15-V}
		\lim_{k\to+\infty}\int_{\tau}^{\tau+T}\widehat{Q}_1^k(t)\d t=0.
	\end{align}
	Making use of Gronwall's inequality to the estimate \eqref{LC11-V}, we get
	\begin{align}\label{LC16-V}
		\|\mathscr{V}^k(t)\|^2_{\V}\leq e^{\int_{\tau}^{\tau+T}\widehat{P}_1^k(t)\d t}\left[\int_{\tau}^{\tau+T}\widehat{Q}_1^k(t)\d t\right].
	\end{align}
	In view of \eqref{LC14-V}-\eqref{LC16-V}, one can complete the proof.
\end{proof}

Lemma \ref{Soln} ensures us that we can define a mapping $\Phi:\R^+\times\R\times\Omega\times\V\to\V$ and Lusin continuity in Proposition \ref{LusinC-V} provides its $\mathscr{F}$-measurability. Consequently, in view of Lemma \ref{Soln} and Proposition \ref{LusinC-V}, the mapping $\Phi$ defined by \eqref{Phi} is an NRDS on $\V$.

\subsubsection{Backward convergence of NRDS}
In this subsection, we prove that the solution to the system \eqref{CNSE-M} converges to the solution of the corresponding autonomous system \eqref{A-CNSE-M} in $\V$ as $\tau\to-\infty$.
\begin{proposition}\label{Back_conver-V}
	Suppose that Hypothesis \ref{Hypo_f-N} is satisfied. Then, the solution $\v$ of the system \eqref{CNSE-M} backward converges to the solution $\widetilde{\v}$ of the system \eqref{A-CNSE-M} in $\V$, that is,
	\begin{align*}
		\lim_{\tau\to -\infty}\|\v(T+\tau,\tau,\theta_{-\tau}\omega,\v_{\tau})-\widetilde{\v}(t,\omega,\widetilde{\v}_0)\|_{\V}=0, \ \ \text{ for all } T>0 \text{ and } \omega\in\Omega,
	\end{align*}
	whenever $\|\v_{\tau}-\widetilde{\v}_0\|_{\V}\to0$ as $\tau\to-\infty.$
\end{proposition}
\begin{proof}
 Taking the inner product with $\A\mathscr{V}^{\tau}(\cdot)$ in \eqref{BC1}, we get
	\begin{align}\label{BC2-V}
		&\frac{1}{2}\frac{\d }{\d t}\|\mathscr{V}^{\tau}\|^2_{\V}\nonumber\\&=-\nu\|\A\mathscr{V}^{\tau}\|^2_{\H}+\sigma z(\theta_t\omega)\|\mathscr{V}^{\tau}\|^2_{\V}-e^{z(\theta_{t}\omega)}b(\mathscr{V}^{\tau},\mathscr{V}^{\tau},\A\mathscr{V}^{\tau})-e^{z(\theta_{t}\omega)}b(\mathscr{V}^{\tau},\widetilde{\v},\A\mathscr{V}^{\tau})\nonumber\\&\quad-e^{z(\theta_{t}\omega)}b(\widetilde{\v},\mathscr{V}^{\tau},\A\mathscr{V}^{\tau})+e^{-z(\theta_{t}\omega)}(\f(t+\tau)-\f_{\infty},\A\mathscr{V}^{\tau}).
	\end{align}
	Using \eqref{b2}, H\"older's and Young's inequalities, we deduce that
		\begin{align}
		\left|e^{z(\theta_{t}\omega)}b(\mathscr{V}^{\tau},\mathscr{V}^{\tau},\A\mathscr{V}^{\tau})\right|&\leq e^{z(\theta_{t}\omega)}\|\mathscr{V}^{\tau}\|^{1/2}_{\H}\|\mathscr{V}^{\tau}\|_{\V}\|\A\mathscr{V}^{\tau}\|^{3/2}_{\H}\nonumber\\&\leq Ce^{4z(\theta_{t}\omega)}\|\mathscr{V}^{\tau}\|^2_{\H}\|\mathscr{V}^{\tau}\|^4_{\V}+\frac{\nu}{8}\|\A\mathscr{V}^{\tau}\|^2_{\V},\label{BC4-V}\\
		\left|e^{z(\theta_{t}\omega)}b(\widetilde{\v},\mathscr{V}^{\tau},\A\mathscr{V}^{\tau})\right|&\leq Ce^{z(\theta_{t}\omega)}\|\widetilde{\v}\|^{1/2}_{\H}\|\widetilde{\v}\|^{1/2}_{\V}\|\mathscr{V}^{\tau}\|^{1/2}_{\V}\|\A\mathscr{V}^{\tau}\|^{3/2}_{\H}\nonumber\\&\leq C e^{4z(\theta_{t}\omega)}\|\widetilde{\v}\|^{4}_{\V}\|\mathscr{V}^{\tau}\|^{2}_{\V}+\frac{\nu}{8}\|\A\mathscr{V}^{\tau}\|^{2}_{\H},\label{BC5-V}\\
		\left|e^{z(\theta_{t}\omega)}b(\mathscr{V}^{\tau},\widetilde{\v},\A\mathscr{V}^{\tau})\right|&\leq Ce^{z(\theta_{t}\omega)}\|\mathscr{V}^{\tau}\|^{1/2}_{\H}\|\mathscr{V}^{\tau}\|^{1/2}_{\V}\|\widetilde{\v}\|^{1/2}_{\V}\|\A\widetilde{\v}\|^{1/2}_{\H}\|\A\mathscr{V}^{\tau}\|_{\H}\nonumber\\&\leq C e^{2z(\theta_{t}\omega)}\|\A\widetilde{\v}\|^{2}_{\H}\|\mathscr{V}^{\tau}\|^{2}_{\V}+\frac{\nu}{8}\|\A\mathscr{V}^{\tau}\|^{2}_{\H},\label{BC6-V}
	\end{align}
and
	\begin{align}\label{BC7-V}
	\hspace{-23mm}	\left|e^{-z(\theta_{t}\omega)}(\f(t+\tau)-\f_{\infty},\A\mathscr{V}^{\tau})\right|\leq C e^{-2z(\theta_{t}\omega)}\|\f(t+\tau)-\f_{\infty}\|^2_{\H}+\frac{\nu}{8}\|\A\mathscr{V}^{\tau}\|^2_{\H}.
	\end{align}
	Combining \eqref{BC2-V}-\eqref{BC7-V}, we arrive at
	\begin{align}\label{BC6V}
		\frac{\d }{\d t}\|\mathscr{V}^{\tau}\|^2_{\H}\leq C\big[\widehat{S}_1(t)\|\mathscr{V}^{\tau}\|^2_{\H}+e^{-2z(\theta_{t}\omega)}\|\f(t+\tau)-\f_{\infty}\|^2_{\H}\big],
	\end{align}
	where
	\begin{align*}
		\widehat{S}_1(t)=e^{4z(\theta_{t}\omega)}\|\mathscr{V}^{\tau}\|^2_{\H}\|\mathscr{V}^{\tau}\|^2_{\V}+e^{4z(\theta_{t}\omega)}\|\widetilde{\v}\|^4_{\V}+e^{2z(\theta_{t}\omega)}\|\A\widetilde{\v}\|^{2}_{\H}+2\left|z(\theta_{t}\omega)\right|.
	\end{align*}
	Making use of Gronwall's inequality in \eqref{BC6V} over $(0,T)$, we obtain
	\begin{align}\label{BC8-V}
		\|\mathscr{V}^{\tau}(T)\|^2_{\V}\leq \left[\|\mathscr{V}^{\tau}(0)\|^2_{\V}+C\int_{0}^{T}e^{-2z(\theta_{t}\omega)}\|\f(t+\tau)-\f_{\infty}\|^2_{\H} \d t\right]e^{C\int_{0}^{T}\widehat{S}_1(t)\d t}.
	\end{align}
	It implies from continuity of $z$, $\widetilde{\v}\in\mathrm{C}([0,T];\V)\cap\mathrm{L}^2(0,T;\D(\A))$, \eqref{BC} and \eqref{BC10} that
	\begin{align}\label{BC10-V}
		\int_{0}^{T}\widehat{S}_{1}(t)\d t<+\infty\  \text{ as } \ \tau\to-\infty.
	\end{align} From Hypothesis \ref{Hypo_f-N}, we deduce that
	\begin{align}\label{BC9-V}
		&\int_{0}^{T}e^{-2z(\theta_{t}\omega)}\|\f(t+\tau)-\f_{\infty}\|^2_{\H} \d t\nonumber\\&\leq\sup_{t\in[0,T]}\left[e^{-2z(\theta_{t}\omega)}\right] \int_{-\infty}^{\tau+T}\|\f(t)-\f_{\infty}\|^2_{\H} \d t\to 0 \ \text{ as } \ \tau\to -\infty.
	\end{align}
	Using \eqref{BC10-V}-\eqref{BC9-V} and $\|\mathscr{V}^{\tau}(0)\|^2_{\V}=\|\v_{\tau}-\widetilde{\v}_0\|_{\V}\to0$ as $\tau\to-\infty$ in \eqref{BC8-V}, one can conclude the proof.
\end{proof}
\subsubsection{Increasing random absorbing set}
This subsection provides the existence of increasing $\mathfrak{D}$-random absorbing set in $\V$ for non-autonomous SNSE \eqref{SNSE} with $S(\u)=\u$.
\begin{proposition}\label{IRAS-V}
	Suppose that $\f\in\mathrm{L}^2_{\mathrm{loc}}(\R;\H)$. Then, there is an increasing random absorbing set $\mathcal{K}_{\V}\subset\V$ given by
	\begin{align}\label{IRAS1-V}
		\mathcal{K}_{\V}(\tau,\omega):=\left\{\u\in\V:\|\u\|^2_{\V}\leq e^{z(\omega)}\widehat{K}(\tau,\omega)\right\}, \ \text{ for all } \ \tau\in\R,
	\end{align}
	where $\widehat{K}(\tau,\omega)$ is defined by \eqref{AB1-V}. Moreover, $\mathcal{K}_{\V}$ is backward tempered, that is, $\mathcal{K}_{\V}\in\mathfrak{D}$.
\end{proposition}
\begin{proof}
	From \eqref{Absorbing-R}, we have $K(\tau,\omega)<+\infty$. It implies from the definition of $\widehat{K}(\tau,\omega)$ that $\widehat{K}(\tau,\omega)<+\infty$. Also, $\mathcal{K}_{\V}$ is tempered. Further, $\mathcal{K}_{\V}(\tau,\omega)$ is increasing due to the fact that $\tau\mapsto K(\tau,\omega)$ is an increasing function. Then, $\mathcal{K}_{\V}$ is an increasing tempered set which gives the backward temperedness of $\mathcal{K}_{\V}(\tau,\omega)$, that is, $\mathcal{K}_{\V}\in\mathfrak{D}$. We have from Lemma \ref{AbsorbingV} that for each $\tau\in\R$, $\omega\in\Omega$ and $D\in\mathfrak{D}$, there exists $\mathcal{T}=\mathcal{T}(\tau,\omega,D)\geq2$ such that for all $t\geq\mathcal{T}$,
	\begin{align*}
		\Phi(t,\tau-t,\theta_{-t}\omega,D(\tau-t,\theta_{-t}\omega))=\u(\tau,\tau-t,\theta_{-\tau}\omega,D(\tau-t,\theta_{-t}\omega))\subseteq\mathcal{K}_{\V}(\tau,\omega),
	\end{align*}
	and  the absorption follows. For the measurability, see Lemma \ref{IRAS}.
\end{proof}
\subsubsection{Backward flattening estimate}
 In order to prove the existence of backward compact random attractors in $\V$, we prove that the cocycle $\Phi$ corresponding to the 2D SNSE \eqref{SNSE} with $S(\u)=\u$ satisfies the backward flattening property in $\V$.

Let $\{e_j\}_{j=1}^{+\infty}\subset\D(\A)$ is the family of eigenfunctions for $\A$ with corresponding positive eigenvalues $\lambda_1\leq\lambda_2\leq\cdots\leq\lambda_j\to+\infty$ as $j\to+\infty$, which form an orthonormal basis of $\H$. Then, $\v\in\H$ has the orthogonal decomposition:
\begin{align*}
	\v=\P_{i}\v\oplus(\I-\P_{i})\v=:\v_{1}+\v_{2},  \ \ \text{ for eaah } i\in\N,
\end{align*}
where, $\P_i:\H\to\H_{i}:=\mathrm{span}\{e_1,e_2,\cdots,e_i\}\subset\H$ is a canonical projection. 	Remember that for $\psi\in\D(\A)$, we have  \begin{align*}\mathrm{P}_i\psi&=\sum_{j=1}^{i}(\psi,e_j)e_j,\  \A\mathrm{P}_i\psi=\sum_{j=1}^{i}\lambda_j(\psi,e_j)e_j,\\ (\mathrm{I}-\mathrm{P}_i)\psi:=\mathrm{Q}_i\psi&=\sum_{j=i+1}^{+\infty}(\psi,e_j)e_j,\ \A\mathrm{Q}_i\psi=\sum_{j=i+1}^{+\infty}\lambda_j(\psi,e_j)e_j, \\
	\|\A\mathrm{Q}_i\psi\|_{\H}^2&=\sum_{j=i+1}^{+\infty}\lambda_j^2|(\psi,e_j)|^2\geq \lambda_{i+1}\sum_{j=m+1}^{+\infty}\lambda_j|(\psi,e_j)|^2=\lambda_{i+1}\|\Q_i\psi\|_{\V}^2,
\end{align*}
and
\begin{align*}
	\|\A\mathrm{P}_i\psi\|_{\H}^2=\sum_{j=1}^{i}\lambda_j^2|(\psi,e_j)|^2\leq \lambda_{i}\sum_{j=1}^{i}\lambda_j|(\psi,e_j)|^2=\lambda_{i}\|\P_i\psi\|_{\V}^2.
\end{align*}
That is, we get \begin{align}\label{4.9}\|\A\Q_i\psi\|_{\H}\geq \sqrt{\lambda_{i+1}}\|\Q_i\psi\|_{\V}\ \text{ and }\ \|\A\P_i\psi\|_{\H}\leq \sqrt{\lambda_i}\|\P_i\psi\|_{\V}.\end{align}
\begin{lemma}\label{Flattening}
	Suppose that $\f\in\mathrm{L}^2_{\mathrm{loc}}(\R;\H)$ and $(\tau,\omega,D)\in\R\times\Omega\times\mathfrak{D}$. Then
	\begin{align}\label{FL-Property}
		\lim_{i,t\to+\infty}\sup_{s\leq \tau}\sup_{\v_{0}\in D(s-t,\theta_{-t}\omega)}\|(\I-\P_{i})\v(s,s-t,\theta_{-s},\v_{\tau,2})\|^2_{\H}=0,
	\end{align}
	where $\v_{\tau,2}=(\I-\P_{i})\v_{\tau}$.
\end{lemma}
\begin{proof}
	Let $\tau\in\R$ be fixed, and $s\leq\tau$. We multiply the first equation of system \eqref{CNSE-M} by $\A\v_2$ and integrating over $\mathcal{O}$, we get
	\begin{align*}
	&	\frac{1}{2}\frac{\d}{\d t}\|\v_2\|^2_{\V}+\|\A\v_2\|^2_{\H}=-e^{z(\theta_{t}\omega)}b(\v,\v,\A\v_2)+e^{-z(\theta_{t}\omega)}(\f,\A\v_2)+\sigma z(\theta_{t}\omega)(\v,\A\v_2)\\  &\implies \frac{1}{2}\frac{\d}{\d t}\|\v_2\|^2_{\V}+\frac{\nu}{2}\|\A\v_2\|^2_{\H}\leq Ce^{4z(\theta_{t}\omega)}\|\v\|^6_{\V}+C\|\A\v\|^2_{\H}+C\big|z(\theta_{t}\omega)\big|^2\|\v\|^2_{\V} \\ &\implies \frac{1}{2}\frac{\d}{\d t}\|\v_2\|^2_{\V}+\frac{\nu\lambda_{i+1}}{2}\|\v_2\|^2_{\V}\leq Ce^{4z(\theta_{t}\omega)}\|\v\|^6_{\V}+C\|\A\v\|^2_{\H}+C\big|z(\theta_{t}\omega)\big|^2\|\v\|^2_{\V}\nonumber\\ &\implies  \frac{\d}{\d t}\big[e^{\nu\lambda_{i+1} t}\|\v_2\|^2_{\V}\big]\leq Ce^{\nu\lambda_{i+1} t}\big[e^{4z(\theta_{t}\omega)}\|\v\|^6_{\V}+\|\A\v\|^2_{\H}+\big|z(\theta_{t}\omega)\big|^2\|\v\|^2_{\V}\big],
	\end{align*}
where we have used \eqref{b2}, \eqref{poin}, Young's inequality and \eqref{4.9}. In view of \eqref{AB1-V} and \eqref{AB1-D(A)}, we apply the uniform Gronwall lemma and obtain the required convergence \eqref{FL-Property}, which completes the proof.
\end{proof}
\subsubsection{Backward compact random attractors and their asymptotic autonomy}
In this subsection, we demonstrate the second result of this section, that is, the existence of a backward compact random attractor and asymptotic autonomy of random attractors in $\V$ for the solution of the system \eqref{SNSE}. For the existence of unique random attractor in $\V$ for autonomous SNSE driven by multiplicative noise on bounded domains, see \cite{CF}. Finally, we prove the main result of this section.

\begin{theorem}\label{MT1-V}
	Suppose that Hypothesis \ref{Hypo_f-N} is satisfied. Then, the NRDS $\Phi$ generated by the system \eqref{SNSE} with $S(\u)=\u$ has a backward compact random attractor $\widetilde{\mathcal{A}}$ in $\V$ such that $\widetilde{\mathcal{A}}$ is a backward tempered set. Furthermore, this backward compact random attractor $\widetilde{\mathcal{A}}$ backward converges to $\widetilde{\mathcal{A}}_{\infty}$, that is,
	\begin{align}\label{MT2-V}
		\lim_{\tau\to -\infty}\mathrm{dist}_{\V}(\widetilde{\mathcal{A}}(\tau,\omega),\widetilde{\mathcal{A}}_{\infty}(\omega))=0\  \emph{ for all } \ \omega\in\Omega,
	\end{align}
	where $\widetilde{\mathcal{A}}_{\infty}$ is the random attractor for the autonomous system \eqref{A-CNSE-M}. For any sequence $\tau_{n}\to-\infty$ as $n\to+\infty$, there is a subsequence $\{\tau_{n_k}\}\subseteq\{\tau_{n}\}$  such that
	\begin{align}\label{MT3-V}
		\lim_{k\to +\infty}\mathrm{dist}_{\V}(\widetilde{\mathcal{A}}(\tau_{n_k},\theta_{\tau_{n_k}}\omega),\widetilde{\mathcal{A}}_{\infty}(\theta_{\tau_{n_k}}\omega))=0\   \emph{ for all } \ \omega\in\Omega.
	\end{align}
\end{theorem}

\begin{proof}
	In order to complete the proof, we use an abstract result proved in \cite{YR} (see Theorem \ref{Abstract-result}). For that purpose, we confirm all the assumptions $(a_1)$, $(a_2)$, $(b_1)$ and $(b_2)$ of Theorem \ref{Abstract-result}.
	\vskip 2mm
	\noindent
	\textit{Verification of $(a_1)$ assumption}: It implies from Proposition \ref{IRAS-V} that NRDS $\Phi$ generated by the system \eqref{SNSE} with $S(\u)=\u$ has  an increasing random absorbing set $\mathcal{K}_{\V}\subset\V$, which is backward tempered, that is, $\mathcal{K}_{\V}\in\mathfrak{D}$. Hence, the condition $(a_1)$  of Theorem \ref{Abstract-result} is verified.
	\vskip 2mm
	\noindent
	\textit{Verification of $(a_2)$ assumption}: Next, we claim that $\Phi$ is $\mathfrak{D}$-backward asymptotically compact in $\V$.
	
	For this purpose, let us fix $(\tau,\omega,D)\in\R\times\Omega\times\mathfrak{D}$ and take arbitrary sequences $s_n\leq\tau$, $\tau_{n}\to+\infty$ and $\v_{0,n}\in D(s_n-t_n,\theta_{-t_n}\omega)$. We prove  that the sequence $$\{\v(s_n,s_n-t_n,\theta_{-s_n}\omega,\v_{0,n})\}_{n\in\N}$$ is precompact. Lemma \ref{AbsorbingV} gives that $\{\P_{i_0}\v(s_n,s_n-t_n,\theta_{-s_n}\omega,\v_{0,n})\}_{n\in\N}$ is bounded in $\V$ which implies its pre-compactness in the $i_0$-dimensional subspace $\V_{i_0}$ of $\V$. Then, there is an indexed  subsequence $n^*$ of  $n$ such that $\{\P_{i_0}\v(s_{n^*},s_{n^*}-t_{n^*},\theta_{-s_{n^*}}\omega,\v_{0,n^*})\}_{n^*\in\N}$ is a Cauchy sequence in $\V_{i_0}$. On the other hand, for each $\varepsilon>0$, let $n^*, m^*$ and $i_0$ be sufficiently large so that
	\begin{align*}
		&\|\v(s_{n^*},s_{n^*}-t_{n^*},\theta_{-s_{n^*}}\omega,\v_{0,n^*})-\v(s_{m^*},s_{m^*}-t_{m^*},\theta_{-s_{m^*}}\omega,\v_{0,m^*})\|_{\V}\nonumber\\&\leq \|\P_{i_0}\v(s_{n^*},s_{n^*}-t_{n^*},\theta_{-s_{n^*}}\omega,\v_{0,n^*})-\P_{i_0}\v(s_{m^*},s_{m^*}-t_{m^*},\theta_{-s_{m^*}}\omega,\v_{0,m^*})\|_{\V}\nonumber\\&\qquad+\|(\I-\P_{i_0})\v(s_{n^*},s_{n^*}-t_{n^*},\theta_{-s_{n^*}}\omega,\v_{0,n^*})\|_{\V}\nonumber\\&\qquad+\|(\I-\P_{i_0})\v(s_{m^*},s_{m^*}-t_{m^*},\theta_{-s_{m^*}}\omega,\v_{0,m^*})\|_{\V}\leq\varepsilon,
	\end{align*}
	where we have used Lemma \ref{Flattening} and the fact that $\{\P_{i_0}\v(s_{n^*},s_{n^*}-t_{n^*},\theta_{-s_{n^*}}\omega,\v_{0,n^*})\}_{n^*\in\N}$ is a Cauchy sequence in $\V_{i_0}$. Therefore, we obtain that $\{\v(s_{n^*},s_{n^*}-t_{n^*},\theta_{-s_{n^*}}\omega,\v_{0,n^*})\}_{n\in\N}$ is a Cauchy sequence in $\V$ which implies that the sequence $\{\v(s_{n^*},s_{n^*}-t_{n^*},\theta_{-s_{n^*}}\omega,\v_{0,n^*})\}_{n\in\N}$ is convergent. Hence, $\Phi$ is $\mathfrak{D}$-backward asymptotically compact in $\V$.
	
	Up to this point, we have shown that the assumptions $(a_1)$ and $(a_2)$ of Theorem \ref{Abstract-result} are satisfied. Therefore, the NRDS $\Phi$ generated by SNSE \eqref{SNSE} has a backward compact random attractor $\widetilde{\mathcal{A}}\in\mathfrak{D}$.
	
	\vskip 2mm
	\noindent
	\textit{Verification of $(b_1)$ assumption}: Proposition \ref{Back_conver-V} gives the backward convergence of the NRDS $\Phi$ to the RDS $\Phi_{\infty}$, that is, the assumption $(b_2)$ of Theorem \ref{Abstract-result} is verified.
	\vskip 2mm
	\noindent
	\textit{Verification of $(b_2)$ assumption}: Now, it only remains to verify the assumption $(b_2)$ of Theorem \ref{Abstract-result}, that is, $\mathcal{K}^{-1}_{\V}\in\mathfrak{D}_{\infty}$, where $\mathcal{K}^{-1}_{\V}(\omega)=\cup_{\tau\leq-1}\mathcal{K}_{\V}(\tau,\omega)$. Since, $\mathcal{K}_{\V}(\tau,\omega)$ is increasing in the $\tau$, which implies that $\mathcal{K}^{-1}_{\V}(\omega)=\mathcal{K}_{\V}(-1,\omega)$. From the definition \eqref{IRAS1-V} of $\mathcal{K}_{\V}$ , we obtain
	\begin{align}\label{K-T1}
		e^{-\frac{\nu\lambda_{1}}{3}t}\|\mathcal{K}^{-1}_{\V}(\theta_{-t}\omega)\|^2_{\H}\leq e^{-\frac{\nu\lambda_{1}}{3}t}\|\mathcal{K}^{-1}_{\V}(\theta_{-t}\omega)\|^2_{\V}\leq e^{-\frac{\nu\lambda_{1}}{3}t}\widehat{K}(-1,\theta_{-t}\omega).
	\end{align}
	It follows from \eqref{K-T} that $e^{-\frac{\nu\lambda_{1}}{3}t}K(-1,\theta_{-t}\omega)\to0$ as $t\to+\infty$. By the definition of $\widehat{K}(\tau,\omega)$ (see \eqref{AB1-V}), we conclude that $e^{-\frac{\nu\lambda_{1}}{3}t}\widehat{K}(-1,\theta_{-t}\omega)\to0$ as $t\to+\infty$ which  along with \eqref{K-T1} implies $\mathcal{K}^{-1}_{\V}\in\mathfrak{D}_{\infty}$. Hence, the convergences \eqref{MT2} and \eqref{MT3} follow from Theorem \ref{Abstract-result} immediately.
\end{proof}

\section{2D SNSE: Additive noise}\label{sec4}\setcounter{equation}{0}
In this section, we consider 2D SNSE driven by additive white noise, that is, $S(\u)$ is independent of $\u$ and establish the asymptotic autonomy of random attractors. Let us consider the 2D SNSE \eqref{SNSE} perturbed by additive white noise for $t\geq \tau,$ $\tau\in\mathbb{R}$ and $S(\u)=\h\in\D(\A)$ as
\begin{equation}\label{SNSE-A}
	\left\{
	\begin{aligned}
		\frac{\d\u(t)}{\d t}+\nu \A\u(t)+\B(\u(t))&=\f(t) +\h(x)\frac{\d \W(t)}{\d t} , \\
		\u(x,\tau)&=\u_{0}(x),	\ \ \ x\in \mathcal{O},
	\end{aligned}
	\right.
\end{equation}
where $\W(t,\omega)$ is the standard scalar Wiener process on the probability space $(\Omega, \mathscr{F}, \mathbb{P})$ (see Section \ref{2.5} above).

Let us define $\v(t,\tau,\omega,\v_{\tau})=\u(t,\tau,\omega,\u_{\tau})-\h(x)z(\theta_{t}\omega)$, where $z$ is defined by \eqref{OU1} and satisfies \eqref{OU2}, and $\u$ is the solution of \eqref{SNSE-A}. Then $\v$ satisfies:
\begin{equation}\label{CNSE-A}
	\left\{
	\begin{aligned}
		\frac{\d\v}{\d t} +\nu \A\v+ \B(\v+\h z)&= \boldsymbol{f} + \sigma\h z -\nu z\A\h , \quad t> \tau,\ \tau\in\R ,\\
		\v(x,\tau)&=\v_{0}(x)=\u_{0}(x)-\h(x)z(\theta_{\tau}\omega), \ \ x\in\mathcal{O},
	\end{aligned}
	\right.
\end{equation}
in $\V'$ (in weak sense).  The following lemma demonstrates  that the system \eqref{CNSE-A} has unique weak and strong solutions.
\begin{lemma}\label{Soln-N}
	Suppose that $\f\in\mathrm{L}^2_{\mathrm{loc}}(\R;\H)$. For each $(\tau,\omega,\v_{\tau})\in\R\times\Omega\times\H$, the system \eqref{CNSE-A} has a unique solution $\v(\cdot,\tau,\omega,\v_{\tau})\in\mathrm{C}([\tau,+\infty);\H)\cap\mathrm{L}^2_{\mathrm{loc}}(\tau,+\infty;\V)$ such that $\v$ is continuous with respect to the  initial data. In addition, for $\v_{\tau}\in \V$, there exists a unique strong solution $\v(\cdot,\tau,\omega,\v_{\tau})\in \mathrm{C}([0, +\infty); \V) \cap \mathrm{L}^{2}_{\mathrm{loc}}(0, +\infty; \D(\A))$ such that $\v$ is continuous with respect to the initial data.
\end{lemma}
\begin{proof}
	One can prove the existence and uniqueness of solution by a standard Faedo-Galerkin approximation method (cf. Theorem 4.5 in \cite{BL}). For continuity with respect to the initial data $\v_{\tau}$, see the proof of Theorem 4.6 in \cite{BL}.
\end{proof}
In \cite{LXY}, authors have established the existence of backward compact random attractors for  stochastic $g$-Navier-Stokes equations driven by small additive noise which covers 2D SNSE \eqref{SNSE-A} also. Therefore, we are not proving the existence of backward compact random attractors for \eqref{SNSE-A} here. If we do not consider the small noise intensity as in \cite{LXY}, we need the following extra assumption on $\h(\cdot)$ to obtain the existence of backward compact random attractors for \eqref{SNSE-A}.
\begin{hypothesis}\label{AonH-N}
	The function \emph{$\h(\cdot)$} satisfies the following condition: there exists a strictly positive constant ${\aleph}$ such that
	\begin{align*}
		|b(\u,\h,\u)|\leq {\aleph}\|\u\|^2_{\H}, \ \  \text{ for all } \ \u\in\H.
	\end{align*}
\end{hypothesis}

\subsection{Asymptotic autonomy of random attractors in $\H$ as well as in $\V$}
In this section, we prove the asymptotic autonomy of backward compact random attractors in $\H$ as well as in $\V$. First, we obtain the results which help us to verify the assumptions of  the abstract result stated in Theorem \ref{Abstract-result}.
\subsubsection{Backward convergence of NRDS}
In this subsection, we first consider the corresponding autonomous system of the non-autonomous system \eqref{SNSE-A} and prove that the solution to the system \eqref{CNSE-A} converges to the solution of the corresponding autonomous system as $\tau\to-\infty$ in $\H$ as well as in $\V$. Consider the autonomous stochastic 2D NSE subjected to additive white noise:
\begin{equation}\label{A-SNSE}
	\left\{
	\begin{aligned}
		\frac{\d\widetilde{\u}(t)}{\d t}+\nu \A\widetilde{\u}(t)+\B(\widetilde{\u}(t))&=\f_{\infty} +\h(x)\frac{\d \W(t)}{\d t}, \\
		\widetilde{\u}(x,0)&=\widetilde{\u}_{0}(x),\ \	x\in \mathcal{O}.
	\end{aligned}
	\right.
\end{equation}
Let $\widetilde{\v}(t,\omega)=\widetilde{\u}(t,\omega)-\h(x)z(\theta_{t}\omega)$. Then, system \eqref{A-SNSE} can be written in the following pathwise deterministic system:
\begin{eqnarray}\label{A-CNSE}
	\left\{
	\begin{aligned}
		\frac{\d\widetilde{\v}(t)}{\d t} +\nu \A\widetilde{\v}(t)+ \B(\widetilde{\v}(t)+\h z(\theta_{t}\omega)) &= {\boldsymbol{f}}_{\infty} + \sigma\h z(\theta_{t}\omega) -\nu z(\theta_{t}\omega)\A\h , \ t> \tau, \tau\in\R ,\\
		\widetilde{\v}(x,0)&=\widetilde{\v}_{0}(x)=\widetilde{\u}_{0}(x)-\h(x)z(\omega), \ \ x\in\mathcal{O},
	\end{aligned}
	\right.
\end{eqnarray}
in $\V'$ (in weak sense).

\begin{proposition}\label{Back_conver-N}
	Suppose that Hypothesis \ref{Hypo_f-N} is satisfied. Then, the solution $\v$ of the system \eqref{CNSE-A} backward converges to the solution $\widetilde{\v}$ of the system \eqref{A-CNSE}, that is,
	\begin{align}\label{BC-A}
		\lim_{\tau\to -\infty}\|\v(T+\tau,\tau,\theta_{-\tau}\omega,\v_{\tau})-\widetilde{\v}(t,\omega,\widetilde{\v}_0)\|_{\H}=0, \ \ \text{ for all } T>0 \text{ and } \omega\in\Omega,
	\end{align}
	whenever $\|\v_{\tau}-\widetilde{\v}_0\|_{\H}\to0$ as $\tau\to-\infty.$
\end{proposition}
\begin{proof}
	Let $\mathscr{V}^{\tau}(t):=\v(t+\tau,\tau,\theta_{-\tau}\omega,\v_{\tau})-\widetilde{\v}(t,\omega,\widetilde{\v}_0)$ for $t\geq0$. From \eqref{CNSE-A} and \eqref{A-CNSE}, we have
	\begin{align}\label{BC1-A}
		\frac{\d\mathscr{V}^{\tau}}{\d t}&=-\nu \A\mathscr{V}^{\tau}-\left[\B\big(\v+\h z(\theta_{t}\omega)\big)-\B\big(\widetilde{\v}+\h z(\theta_{t}\omega)\big)\right]+\left[\f(t+\tau)-\f_{\infty}\right],
	\end{align}
	in $\V'$ (in weak sense). In view of \eqref{BC1-A}, we obtain
	\begin{align}\label{BC2-A}
		\frac{1}{2}\frac{\d }{\d t}\|\mathscr{V}^{\tau}\|^2_{\H}&=-\nu\|\mathscr{V}^{\tau}\|^2_{\V}-\left\langle\B\big(\v+\h z(\theta_{t}\omega)\big)-\B\big(\widetilde{\v}+\h z(\theta_{t}\omega)\big),\v-\widetilde{\v}\right\rangle\nonumber\\&\quad+(\f(t+\tau)-\f_{\infty},\mathscr{V}^{\tau}).
	\end{align}
	Applying \eqref{b0}, \eqref{441}, H\"older's and Young's inequalities, we infer
	\begin{align}\label{BC4-A}
		&\left|\left\langle\B\big(\v+\h z(\theta_{t}\omega)\big)-\B\big(\widetilde{\v}+\h z(\theta_{t}\omega)\big),\v-\widetilde{\v}\right\rangle\right|\nonumber\\&=\left|\left\langle\B\big(\v+\h z(\theta_{t}\omega)\big)-\B\big(\widetilde{\v}+\h z(\theta_{t}\omega)\big),(\v+\h z(\theta_{t}\omega))-(\widetilde{\v}+\h z(\theta_{t}\omega))\right\rangle\right|\nonumber\\&=\left|b(\mathscr{V}^{\tau},\mathscr{V}^{\tau},\widetilde{\v}+\h z(\theta_{t}\omega))\right|\nonumber\\&\leq C\|\widetilde{\v}+\h z(\theta_{t}\omega)\|^2_{\V}\|\mathscr{V}^{\tau}\|^2_{\H}+\frac{\nu}{4}\|\mathscr{V}^{\tau}\|^2_{\V},
	\end{align}
	and
	\begin{align}\label{BC5-A}
		\left|(\f(t+\tau)-\f_{\infty},\mathscr{V}^{\tau})\right|&\leq\frac{1}{\nu\lambda_{1}}\|\f(t+\tau)-\f_{\infty}\|^2_{\H}+\frac{\nu\lambda_{1}}{4}\|\mathscr{V}^{\tau}\|^2_{\H}\nonumber\\&\leq\frac{1}{\nu\lambda_{1}}\|\f(t+\tau)-\f_{\infty}\|^2_{\H}+\frac{\nu}{4}\|\mathscr{V}^{\tau}\|^2_{\V},
	\end{align}
	where we have used \eqref{poin} in the final inequality. Combining \eqref{BC2-A}-\eqref{BC5-A}, we arrive at
	\begin{align}\label{BC6-A}
		\frac{\d }{\d t}\|\mathscr{V}^{\tau}\|^2_{\H}+\nu\|\mathscr{V}^{\tau}\|^2_{\V}\leq C\|\widetilde{\v}+\h z(\theta_{t}\omega)\|^2_{\V}\|\mathscr{V}^{\tau}\|^2_{\H}+\frac{1}{\nu\lambda_{1}}\|\f(t+\tau)-\f_{\infty}\|^2_{\H}.
	\end{align}
	Making use of Gronwall's inequality in \eqref{BC6-A} over $(0,T)$, we obtain
	\begin{align*}
		\|\mathscr{V}^{\tau}(T)\|^2_{\H}\leq \left[\|\mathscr{V}^{\tau}(0)\|^2_{\H}+\frac{1}{\nu\lambda_{1}}\int_{0}^{T}\|\f(t+\tau)-\f_{\infty}\|^2_{\H} \d t\right]e^{C\int_{0}^{T}\|\widetilde{\v}+\h z(\theta_{t}\omega)\|^2_{\V}\d t}.
	\end{align*}
	Since $\h\in\D(\A)$, $z$ is continuous and $\widetilde{\v}\in\mathrm{L}^2(0,T;\V)$, it implies that
	\begin{align}\label{BC7-A}
		\int_{0}^{T}C\|\widetilde{\v}+\h z(\theta_{t}\omega)\|^2_{\V}\d t<+\infty.	
	\end{align}
	From Hypothesis \ref{Hypo_f-N}, we deduce that
	\begin{align}\label{BC8-A}
		\int_{0}^{T}\|\f(t+\tau)-\f_{\infty}\|^2_{\H} \d t\leq \int_{-\infty}^{\tau+T}\|\f(t)-\f_{\infty}\|^2_{\H} \d t\to 0 \ \text{ as } \ \tau\to -\infty.
	\end{align}
	Using \eqref{BC7-A}-\eqref{BC8-A} and $\|\mathscr{V}^{\tau}(0)\|^2_{\H}=\|\v_{\tau}-\widetilde{\v}_0\|_{\H}\to0$ as $\tau\to-\infty$, we conclude the proof. Furthermore, integrating \eqref{BC6-A} from $0$ to $T$, we get
	\begin{align*}
		&\int_{0}^{T}\|\mathscr{V}^{\tau}(t)\|^2_{\V}\d t\nonumber\\&\leq \frac{1}{\nu}\|\mathscr{V}^{\tau}(0)\|^2_{\H}+C\sup_{t\in[0,T]}\|\mathscr{V}^{\tau}(t)\|^2_{\H}\int_{0}^{T}\|\widetilde{\v}+\h z(\theta_{t}\omega)\|^2_{\V}\d t+\frac{1}{\nu^2\lambda_{1}}\int_{0}^{T}\|\f(t+\tau)-\f_{\infty}\|^2_{\H}\d t.
	\end{align*}
Using \eqref{BC-A}, \eqref{BC7-A}, \eqref{BC8-A} and the fact that $\|\mathscr{V}^{\tau}(0)\|^2_{\H}=\|\v_{\tau}-\widetilde{\v}_0\|_{\H}\to0$ as $\tau\to-\infty$, we conclude
\begin{align}\label{BC9-A}
	\int_{0}^{T}\|\mathscr{V}^{\tau}(t)\|^2_{\V}\d t \to 0 \text{ as } \tau \to -\infty,
\end{align}
for all $T>0$.
\end{proof}

\begin{proposition}\label{Back_conver-VA}
	Suppose that Hypothesis \ref{Hypo_f-N} is satisfied. Then, the solution $\v$ of the system \eqref{CNSE-A} backward converges to the solution $\widetilde{\v}$ of the system \eqref{A-CNSE} in $\V$, that is,
	\begin{align*}
		\lim_{\tau\to -\infty}\|\v(T+\tau,\tau,\theta_{-\tau}\omega,\v_{\tau})-\widetilde{\v}(t,\omega,\widetilde{\v}_0)\|_{\V}=0, \ \ \text{ for all } T>0 \text{ and } \omega\in\Omega,
	\end{align*}
	whenever $\|\v_{\tau}-\widetilde{\v}_0\|_{\V}\to0$ as $\tau\to-\infty.$
\end{proposition}
\begin{proof}
	Taking the inner product with $\A\mathscr{V}^{\tau}(\cdot)$ in \eqref{BC1-A}, we get
	\begin{align}\label{BC2-VA}
		\frac{1}{2}\frac{\d }{\d t}\|\mathscr{V}^{\tau}\|^2_{\V}&=-\nu\|\A\mathscr{V}^{\tau}\|^2_{\H}-b(\mathscr{V}^{\tau},\mathscr{V}^{\tau},\A\mathscr{V}^{\tau})-b(\mathscr{V}^{\tau},\widetilde{\v}+\h z(\theta_t\omega),\A\mathscr{V}^{\tau})\nonumber\\&\quad-b(\widetilde{\v}+\h z(\theta_t\omega),\mathscr{V}^{\tau},\A\mathscr{V}^{\tau})+(\f(t+\tau)-\f_{\infty},\A\mathscr{V}^{\tau}).
	\end{align}
	Using \eqref{b2}, H\"older's and Young's inequalities, we achieve
	\begin{align}
		\left|b(\mathscr{V}^{\tau},\mathscr{V}^{\tau},\A\mathscr{V}^{\tau})\right|&\leq \|\mathscr{V}^{\tau}\|^{\frac{1}{2}}_{\H}\|\mathscr{V}^{\tau}\|_{\V}\|\A\mathscr{V}^{\tau}\|^{\frac{3}{2}}_{\H}\nonumber\\&\leq C\|\mathscr{V}^{\tau}\|^2_{\H}\|\mathscr{V}^{\tau}\|^4_{\V}+\frac{\nu}{8}\|\A\mathscr{V}^{\tau}\|^2_{\V},\label{BC4-VA}\\
		\left|b(\widetilde{\v}+\h z(\theta_t\omega),\mathscr{V}^{\tau},\A\mathscr{V}^{\tau})\right|&\leq C\|\widetilde{\v}+\h z(\theta_t\omega)\|^{\frac{1}{2}}_{\H}\|\widetilde{\v}+\h z(\theta_t\omega)\|^{\frac{1}{2}}_{\V}\|\mathscr{V}^{\tau}\|^{\frac{1}{2}}_{\V}\|\A\mathscr{V}^{\tau}\|^{3/2}_{\H}\nonumber\\&\leq C \|\widetilde{\v}+\h z(\theta_t\omega)\|^{4}_{\V}\|\mathscr{V}^{\tau}\|^{2}_{\V}+\frac{\nu}{8}\|\A\mathscr{V}^{\tau}\|^{2}_{\H},\label{BC5-VA}\\
		\left|b(\mathscr{V}^{\tau},\widetilde{\v}+\h z(\theta_t\omega),\A\mathscr{V}^{\tau})\right|&\leq C\|\mathscr{V}^{\tau}\|^{\frac{1}{2}}_{\H}\|\mathscr{V}^{\tau}\|^{\frac{1}{2}}_{\V}\|\widetilde{\v}+\h z(\theta_t\omega)\|^{\frac{1}{2}}_{\V}\|\A(\widetilde{\v}+\h z(\theta_t\omega))\|^{\frac{1}{2}}_{\H}\|\A\mathscr{V}^{\tau}\|_{\H}\nonumber\\&\leq C \|\A(\widetilde{\v}+\h z(\theta_t\omega))\|^{2}_{\H}\|\mathscr{V}^{\tau}\|^{2}_{\V}+\frac{\nu}{8}\|\A\mathscr{V}^{\tau}\|^{2}_{\H},\label{BC6-VA}
	\end{align}
	and
	\begin{align}\label{BC7-VA}
	\hspace{-40mm}	\left|(\f(t+\tau)-\f_{\infty},\A\mathscr{V}^{\tau})\right|\leq C\|\f(t+\tau)-\f_{\infty}\|^2_{\H}+\frac{\nu}{8}\|\A\mathscr{V}^{\tau}\|^2_{\H}.
	\end{align}
	Combining \eqref{BC2-VA}-\eqref{BC7-VA}, we obtain
	\begin{align*}
		\frac{\d }{\d t}\|\mathscr{V}^{\tau}\|^2_{\H}\leq C\big[
		\widehat{S}_2(t)\|\mathscr{V}^{\tau}\|^2_{\H}+\|\f(t+\tau)-\f_{\infty}\|^2_{\H}\big],
	\end{align*}
	where
	\begin{align*}
		\widehat{S}_2(t)=C\|\mathscr{V}^{\tau}\|^2_{\H}\|\mathscr{V}^{\tau}\|^2_{\V}+\|\widetilde{\v}+\h z(\theta_{t}\omega)\|^4_{\V}+\|\A(\widetilde{\v}+\h z(\theta_{t}\omega))\|^2_{\V}.
	\end{align*}
	Making use of Gronwall's inequality in \eqref{BC6} over $(0,T)$, we get
	\begin{align}\label{BC9-VA}
		\|\mathscr{V}^{\tau}(T)\|^2_{\V}\leq \left[\|\mathscr{V}^{\tau}(0)\|^2_{\V}+C\int_{0}^{T}\|\f(t+\tau)-\f_{\infty}\|^2_{\H} \d t\right]e^{C\int_{0}^{T}\widehat{S}_2(t)\d t}.
	\end{align}
	It implies from $\h\in\D(\A)$, continuity of $z$, $\widetilde{\v}\in\mathrm{C}([0,T];\V)\cap\mathrm{L}^2(0,T;\D(\A))$, convergences \eqref{BC-A} and \eqref{BC9-A} that
	\begin{align}\label{BC10-VA}
		\int_{0}^{T}\widehat{S}_{2}(t)\d t<+\infty \ \text{ as } \ \tau\to-\infty.
	\end{align}	
	Using \eqref{BC10-VA}, \eqref{BC9} and $\|\mathscr{V}^{\tau}(0)\|^2_{\V}=\|\v_{\tau}-\widetilde{\v}_0\|_{\V}\to0$ as $\tau\to-\infty$ in \eqref{BC9-VA}, we complete the proof.
\end{proof}

\subsubsection{Increasing random absorbing sets}
This subsection provides the existence of increasing ${\mathfrak{D}}$-random absorbing set for non-autonomous stochastic NSE. We are only stating the following two propositions which show the existence of increasing $\mathfrak{D}$-random absorbing sets in $\H$ as well as in $\V$. For  proofs of the following two propositions, we refer readers to \cite{LXY}.

\begin{proposition}\label{IRAS-N}
	Suppose that $\f\in\mathrm{L}^2_{\mathrm{loc}}(\R;\H)$ and Hypothesis \ref{AonH-N} is satisfied. Then, there is an increasing random absorbing set $\mathcal{R}_{\H}\in\mathfrak{D}$ given by
	\begin{align}\label{IRAS1-N}
		\mathcal{R}_{\H}(\tau,\omega)=\left\{\u\in\H:\|\u\|^2_{\H}\leq R_{\H}(\tau,\omega)\right\}, \text{ for all } \tau\in\R,
	\end{align}
where
	\begin{align}\label{R1}
	R_{\H}(\tau,\omega)&= 2+ C\sup_{s\leq \tau}\int_{-\infty}^{0}e^{\nu\lambda\uprho-4{\aleph}\int^{\uprho}_{0}\left|z(\theta_{\upeta}\omega)\right|\d\upeta}\|\f(\uprho+s)\|^2_{\H}\d\uprho\nonumber\\&\quad\quad+C\int_{-\infty}^{0}e^{\nu\lambda\uprho-4{\aleph}\int^{\uprho}_{0}\left|z(\theta_{\upeta}\omega)\right|\d\upeta}\bigg\{\left|z(\theta_{\uprho}\omega)\right|^2+\left|z(\theta_{\uprho}\omega)\right|^{3}\bigg\}\d\uprho+2\|\h\|^2_{\H}\left|z(\omega)\right|^2\nonumber\\&\quad=: 2+\rho_1(\tau,\omega)+2\|\h\|^2_{\H}\left|z(\omega)\right|^2.
\end{align}
\end{proposition}
\begin{proof}
See the proof of Lemmas 4.1 and 5.1 in \cite{LXY}.
\end{proof}

\begin{proposition}\label{IRAS-NV}
	Suppose that $\f\in\mathrm{L}^2_{\mathrm{loc}}(\R;\H)$ and Hypothesis \ref{AonH-N} is satisfied. Then, there is an increasing random absorbing set $\mathcal{R}_{\V}$ given by
	\begin{align}\label{IRAS1-NV}
		\mathcal{R}_{\V}(\tau,\omega)=\left\{\u\in\H:\|\u\|^2_{\H}\leq R_{\V}(\tau,\omega)\right\}, \text{ for all } \tau\in\R,
	\end{align}
where
\begin{align}\label{R2}
	R_{\V}(\tau,\omega)&=2+Ce^{\sup\limits_{\xi\in[-2,0]}\left[\rho_1(\tau,\omega)+|z(\theta_{\xi}\omega)|^2\right]\rho_2(\tau,\omega)}\bigg[\rho_2(\tau,\omega)+C\sup_{\xi\in[-2,0]}\big\{\rho_1(\tau,\omega)\left|z(\theta_{\xi}\omega)\right|^4\nonumber\\&\quad\quad+\left|z(\theta_{\xi}\omega)\right|^6+\rho_1(\tau,\omega)+1\big\}\bigg]+2\|\h\|^2_{\V}\left|z(\omega)\right|^2,
\end{align}
and
\begin{align}\label{R3}
	\rho_2(\tau,\omega)=C\rho_1(\tau,\omega)+C\rho_1(\tau,\omega)\int_{-2}^{0}\left|z(\theta_{\xi}\omega)\right|\d\xi+C\int_{-2}^{0}\left\{\left|z(\theta_{\xi}\omega)\right|^2+\left|z(\theta_{\xi}\omega)\right|^4\right\}\d\xi.
\end{align}
\end{proposition}
\begin{proof}
	See the proof of Lemmas 4.2 and 5.1 in \cite{LXY}.
\end{proof}

\subsubsection{Asymptotic autonomy of random attractors}
In this subsection, we demonstrate the main result of this section, that is, asymptotic autonomy of backward compact random attractors for the solution of the system \eqref{SNSE-A}. For the existence of a unique random attractor for autonomous and non-autonomous SNSE driven by additive noise on bounded domains in $\H$ as well as $\V$, see \cite{FY} and \cite{LXY}, respectively. Finally, we prove the main result of this section.
\begin{theorem}\label{MT1-N}
	Consider $\X=\H$ or $\V$. Suppose that Hypotheses \ref{Hypo_f-N} and \ref{AonH-N} are satisfied. Then, the backward compact random attractor $\mathscr{A}$ backward converges to $\mathscr{A}_{\infty}$, that is,
	\begin{align}\label{MT2-N}
		\lim_{\tau\to -\infty}\mathrm{dist}_{\X}(\mathscr{A}(\tau,\omega),\mathscr{A}_{\infty}(\omega))=0, \ \emph{ for all }\  \omega\in\Omega,
	\end{align}
	where $\mathscr{A}$ and $\mathscr{A}_{\infty}$ are the random attractor for the systems \eqref{SNSE} and \eqref{A-SNSE}, respectively. Moreover, for any sequence $\tau_{n}\to-\infty$ as $n\to+\infty$, there is a subsequence $\{\tau_{n_k}\}\subseteq\{\tau_{n}\}$  such that
	\begin{align}\label{MT3-N}
		\lim_{k\to +\infty}\mathrm{dist}_{\X}(\mathscr{A}(\tau_{n_k},\theta_{\tau_{n_k}}\omega),\mathscr{A}_{\infty}(\theta_{\tau_{n_k}}\omega))=0\ \emph{ for all }\  \omega\in\Omega.
	\end{align}
\end{theorem}
\begin{proof}
	In order to complete the proof, we use an abstract result proved in \cite{YR} (see Theorem \ref{Abstract-result}). For this, we establish that the assumptions $(b_1)$ and $(b_2)$ of Theorem \ref{Abstract-result} are verified.
	\vskip 2mm
	\noindent
	\textit{Verification of $(b_1)$ assumption}: It is immediate from Propositions \ref{Back_conver-N} and \ref{Back_conver-VA}  that the assumption $(b_1)$  is satisfied for $\X=\H$ and $\X=\V$, respectively.
	\vskip 2mm
	\noindent
	\textit{Verification of $(b_2)$ assumption}: Finally, it only remains to prove that the assumption $(b_2)$  of Theorem \ref{Abstract-result} is satisfied, that is, $\mathcal{R}^{-1}_{\X}\in\mathfrak{D}_{\infty}$, where $\mathcal{R}^{-1}_{\X}(\omega)=\cup_{\tau\leq-1}\mathcal{R}_{\X}(\tau,\omega)$. Note that  $\mathcal{R}_{\X}(\tau,\omega)$ is increasing in the $\tau$  implies that $\mathcal{R}^{-1}_{\X}(\omega)=\mathcal{R}_{\X}(-1,\omega)$. From the definition \eqref{IRAS1-N} of $\mathcal{R}_{\H}$, we obtain
	\begin{align}\label{AV}
		e^{-\frac{\nu\lambda_{1}}{3}t}\|\mathcal{R}^{-1}_{\H}(\theta_{-t}\omega)\|^2_{\H}&\leq e^{-\frac{\nu\lambda_{1}}{3}t}R_{\H}(-1,\theta_{-t}\omega)\nonumber\\&\leq 2e^{-\frac{\nu\lambda_{1}}{3}t}+ e^{-\frac{\nu\lambda_{1}}{3}t}\rho_1(-1,\theta_{-t}\omega)+2e^{-\frac{\nu\lambda_{1}}{3}t}\|\h\|^2_{\H}|z(\theta_{-t}\omega)|.
	\end{align}
Now, we consider $\sigma$ large enough $\left(\sigma\geq\frac{36\aleph^2}{\pi\nu^2\lambda_{1}^2}\right)$ such that from \eqref{Z2}, we have
\begin{align}\label{Z6}
	4\aleph\mathbb{E}(|\z(\cdot)|)\leq\frac{2\nu\lambda_{1}}{3}.
\end{align}
From \eqref{Z6}, we have
\begin{align}\label{AV1}
&e^{-\frac{\nu\lambda_{1}}{3}t}\rho_1(-1,\theta_{-t}\omega)\nonumber\\&=Ce^{-\frac{\nu\lambda_{1}}{3}t}\sup_{s\leq-1}\int_{-\infty}^{0}e^{\nu\lambda_{1}\uprho-4{\aleph}\int^{\uprho}_{0}\left|z(\theta_{\upeta-t}\omega)\right|\d\upeta}\|\f(\uprho+s)\|^2_{\H}\d\uprho\nonumber\\&\quad+Ce^{-\frac{\nu\lambda_{1}}{3}t}\int_{-\infty}^{0}e^{\nu\lambda\uprho-4{\aleph}\int^{\uprho}_{0}\left|z(\theta_{\upeta-t}\omega)\right|\d\upeta}\bigg\{\left|z(\theta_{\uprho-t}\omega)\right|^2+\left|z(\theta_{\uprho-t}\omega)\right|^{3}\bigg\}\d\uprho\nonumber\\&=Ce^{-\frac{\nu\lambda_{1}}{3}t}e^{\nu\lambda_{1} t-4\aleph\int_{-t}^{0}|z(\theta_{\upeta}\omega)|\d\upeta}\sup_{s\leq-1}\int_{-\infty}^{s}e^{\nu\lambda_{1}(\uprho-s-t)-4{\aleph}\int_{\uprho-s-t}^{0}\left|z(\theta_{\upeta}\omega)\right|\d\upeta}\|\f(\uprho)\|^2_{\H}\d\uprho\nonumber\\&\quad+Ce^{-\frac{\nu\lambda_{1}}{3}t}e^{\nu\lambda_{1} t-4\aleph\int_{-t}^{0}|z(\theta_{\upeta}\omega)|\d\upeta}\int_{-\infty}^{-t}e^{\nu\lambda_{1}\uprho-4{\aleph}\int^{0}_{\uprho}\left|z(\theta_{\upeta}\omega)\right|\d\upeta}\bigg\{\left|z(\theta_{\uprho}\omega)\right|^2+\left|z(\theta_{\uprho}\omega)\right|^{3}\bigg\}\d\uprho\nonumber\\&\leq Ce^{-\frac{\nu\lambda_{1}}{3}t}\sup_{s\leq-1}\int_{-\infty}^{s}e^{\frac{\nu\lambda_{1}}{3}(\uprho-s)}\|\f(\uprho)\|^2_{\H}\d\uprho+C\int_{-\infty}^{-t}e^{\frac{\nu\lambda_{1}}{3}\uprho}\bigg\{\left|z(\theta_{\uprho}\omega)\right|^2+\left|z(\theta_{\uprho}\omega)\right|^{3}\bigg\}\d\uprho\nonumber\\&\to 0\text{ as } t\to+\infty,
\end{align}
	where we have used \eqref{Z5} and \eqref{f2-N}. From \eqref{Z5}, \eqref{AV} and \eqref{AV1}, we get
	\begin{align*}
		e^{-\frac{\nu\lambda_{1}}{3}t}\|\mathcal{R}^{-1}_{\H}(\theta_{-t}\omega)\|^2_{\H}\to 0 \text{ as } t\to +\infty,
	\end{align*}
	which conclude that $\mathcal{R}^{-1}_{\H}\in\mathfrak{D}_{\infty}$. Hence, the convergences \eqref{MT2-N} and \eqref{MT3-N} follow for $\X=\H$ from Theorem \ref{Abstract-result} immediately.
	\vskip2mm
	
	Next, from the definition \eqref{IRAS1-NV} of $\mathcal{R}_{\V}$, we obtain
	\begin{align}\label{K-T1A}
		e^{-\frac{\nu\lambda_{1}}{3}t}\|\mathcal{R}^{-1}_{\V}(\theta_{-t}\omega)\|^2_{\H}\leq e^{-\frac{\nu\lambda_{1}}{3}t}\|\mathcal{R}^{-1}_{\V}(\theta_{-t}\omega)\|^2_{\V}\leq e^{-\frac{\nu\lambda_{1}}{3}t}R_{\V}(-1,\theta_{-t}\omega).
	\end{align}
	It follows from \eqref{AV1} that $e^{-\frac{\nu\lambda_{1}}{3}t}\rho_1(-1,\theta_{-t}\omega)\to0$ as $t\to+\infty$. By the definition of $R_{\V}(\tau,\omega)$ (see \eqref{R2}), we conclude that $e^{-\frac{\nu\lambda_{1}}{3}t}R_{\V}(-1,\theta_{-t}\omega)\to0$ as $t\to+\infty,$ which along with \eqref{K-T1A} implies $\mathcal{R}^{-1}_{\V}\in\mathfrak{D}_{\infty}$. Hence, the convergences \eqref{MT2} and \eqref{MT3} follow for $\X=\V$ from Theorem \ref{Abstract-result} immediately.
\end{proof}

	\medskip\noindent
{\bf Acknowledgments:}    The first author would like to thank the Council of Scientific $\&$ Industrial Research (CSIR), India for financial assistance (File No. 09/143(0938)/2019-EMR-I).  M. T. Mohan would  like to thank the Department of Science and Technology (DST), Govt of India for Innovation in Science Pursuit for Inspired Research (INSPIRE) Faculty Award (IFA17-MA110).

\end{document}